\documentclass[12pt, reqno]{amsart}
\usepackage{amsmath, amsthm, amscd, amsfonts, amssymb, mathrsfs, graphicx, color,eucal}
\usepackage{enumerate}
\usepackage[all]{xy}
\usepackage[bookmarksnumbered, colorlinks, plainpages]{hyperref}
\hypersetup{colorlinks=true,linkcolor=red, anchorcolor=green, citecolor=cyan, urlcolor=red, filecolor=magenta, pdftoolbar=true}

\newtheorem{theorem}{Theorem}[section]
\newtheorem{lemma}[theorem]{Lemma}
\newtheorem{proposition}[theorem]{Proposition}
\newtheorem{corollary}[theorem]{Corollary}
\theoremstyle{definition}
\newtheorem{definition}[theorem]{Definition}
\newtheorem{example}[theorem]{Example}

\theoremstyle{remark}
\newtheorem{remark}[theorem]{Remark}
\numberwithin{equation}{section}

\begin{document}

\setcounter{page}{1}

\title[Algebras of GSIOs]{Algebras of Generalized Singular Integral Operators with Cauchy kernel}

\author[Y.Sang]{Yuanqi Sang}

\address{ 401331, P.R. China.}
\email{\textcolor[rgb]{0.00,0.00,0.84}{yqisang@163.com}}


\let\thefootnote\relax\footnote{}



\begin{abstract}
For bounded Lebesgue measurable functions $f,g,\phi$ and $\psi$ on the unit circle, $P_{+}fP_{+}+P_{-}gP_{+}
+P_{+}\phi P_{-}+P_{-}\psi P_{-}$ is called a generalized singular integral operator (GSIO) on $L^{2}(\mathbb{T})$,
where $P_{+}$ is the Riesz projection, $P_{-}=I-P_{+}.$
In this paper, we relate  GSIOs to
a number of  operators, including Cauchy singular integral operator, (dual) truncated Toeplitz operator, Foguel-Hankel operator, multiplication operator, Toeplitz plus Hankel operator etc.
We establish  the short exact sequences associated  of the $C^{*}-$algebras
generated by GSIOs with bounded or quasi-continuous symbols. As a consequence we obtain  the spectra of various classes of GSIOs, the spectral inclusion theorem and comput the Fredholm index of GSIOs. Moreover, we gave the necessary and sufficient conditions for invertibility(Fredholmness) of  GSIOs via Winer-Hopf factorization.
\end{abstract} \maketitle

\section{Introduction}
Let $\mathbb{D}=\{\xi\in\mathbb{C}:|\xi|<1\}$ be the unit disk in the complex plane $\mathbb{C}$ and  $\mathbb{T}=\{\xi\in\mathbb{C}:|\xi|=1\}$ be its boundary.
Riemann-Hilbert boundary problem \cite{nikolski2020} on the unit circle can be reformulated as follows.

Given functions
$\alpha,\beta,h$ on $\mathbb{T},$ find two analytic functions
$f_{+}\in \text{Hol}(\mathbb{D})$ and $f_{-}\in \text{Hol}(\mathbb{C}\setminus\overline{\mathbb{D}})(f_{-}(\infty)=0)$ such that
\begin{align}\label{rh}
\alpha f_{+}+\beta f_{-}= h
\end{align}
on $\mathbb{T}.$

$H^{2}$ denotes the classical Hardy space of the open unit disk $\mathbb{D}$,
we let $L^{2}=L^{2}(\mathbb{T}),L^{\infty}=L^{\infty}(\mathbb{T})$ denote the usual Lebesgue spaces
on the unit circle \cite{Garnett2007}.
$P_{+}$ is the orthogonal projection from $L^{2}(\mathbb{T})$ onto $H^{2},P_{-}=I-P_{+}.$
Suppose that $h\in L^{2}(\mathbb{T}),f_{+}\in H^{2}$ and
$f_{-}\in L^{2}(\mathbb{T})\ominus H^{2}=\bar{z}\overline{H^{2}}.$
Put
$f=f_{+}+f_{-},$ the equation \eqref{rh}  becomes
\begin{align*}
S_{\alpha,\beta}f=h,\quad \text{where}~
S_{\alpha,\beta}=\alpha P_{+}+\beta P_{-}.
\end{align*}
$S_{\alpha,\beta}$ is called the singular integral operator with Cauchy kernel on $L^{2}(\mathbb{T}),$ and
\begin{align*}
\left(S_{\alpha, \beta} f\right)(z)=\frac{\alpha(z)+\beta(z)}{2} f(z)+\frac{\alpha(z)-\beta(z)}{2} \frac{1}{\pi i} \int_{\mathbb{T}} \frac{f(\xi )}{\xi-z} d \xi.
\end{align*}

Riemann-Hilbert boundary problem is considered solved if one has found conditions
for the operator $S_{\alpha,\beta}$ to be Fredholm or invertible.
Most results about these operators can be found in  \cite{Gohberg1992,Gohberg1992a}.
We are interested in the algebra of singular integral operator, but the adjoint of
$R_{\alpha,\beta}$ is no longer a singular integral operator.
Naturally, one can define the generalized singular
integral operator.

Given a linear space $X,$ we denote by $X_{N}$ the linear space of all $N-$dimensional vectors with components from $X$ and let $X_{N\times N}$ denote the linear space of
$N\times N$ matrices with entries from $X.$
\begin{definition}
If $H={\begin{pmatrix}
\begin{smallmatrix}
   f& \phi\\
   g & \psi \end{smallmatrix}
\end{pmatrix}}\in L_{2\times2}^{2}(\mathbb{T}),$ the generalized singular
integral operator (GSIO) with symbol $H$ is the operator $R_{H}$  is  defined by
\begin{align*}
R_{H}x
=P_{+}fP_{+}x+P_{-}gP_{+}x
+P_{+}\phi P_{-}x+P_{-}\psi P_{-}x.
\end{align*}
for each $x\in L^{2}(\mathbb{T}).$
\end{definition}

The significance of GSIOs comes from the following special cases.
\begin{enumerate}
  \item Multiplication operator
on $L^{2}(\mathbb{T}):$ if $f=g=\phi=\psi,$ then $R_{H}$ is the multiplication operator on $L^{2}(\mathbb{T}).$
\item Hilbert transform: if $f=g=-\phi=-\psi=1,$ then
$R_{\begin{pmatrix}
\begin{smallmatrix}
 1& -1\\
   1 & -1
\end{smallmatrix}
\end{pmatrix}}-1\otimes 1=P_{+}-P_{-}-1\otimes 1$
is the Hilbert transform $\mathbb{H}$ \cite[Ch III]{Garnett2007}.
\item Singular integral operator: if $f=g=\alpha,\phi=\psi=\beta,$ then
$R_{\begin{pmatrix}
\begin{smallmatrix}
 \alpha& \beta\\
   \alpha & \beta
\end{smallmatrix}
\end{pmatrix}}
=S_{\alpha,\beta}$ is the singular integral operator.
T. Nakazi and T. Yamamoto \cite{Nakazi1998,Nakazi1999,Nakazi2000,Nakazi2003,Nakazi2010,Nakazi2014} have study the boundedness and normality of $S_{\alpha,\beta},$ and calculate its norm,
C. Gu \cite{Gu2016} have study the algebraic properties of $S_{\alpha,\beta}.$
\item Toeplitz plus Hankel operators: if $H=\begin{pmatrix}
\begin{smallmatrix}
 f & 0 \\
g & 0
\end{smallmatrix}
\end{pmatrix},$ then $(I\oplus J)R_{H}|_{H^{2}}=T_{f}+\Gamma_{g},$
where $Jx(z)=\bar{z}x(\bar{z})$ for $x\in L^{2}(\mathbb{T}).$
\item Foguel-Hankel operators:
if $\phi\in L^{\infty} $ and $H=\begin{pmatrix}
\begin{smallmatrix}
 \bar{z} & \phi \\
0 & \bar{z}
\end{smallmatrix}
\end{pmatrix},$ then
$R_{H}$
and Foguel-Hankel operator
$\begin{pmatrix}
\begin{smallmatrix}T_{z}^{*}& X\\
		0 &   T_{z}\end{smallmatrix}
\end{pmatrix}$
are unitarily equivalent(see Section 2). Foguel-Hankel  operators closely related to Halmos' problem\cite{Halmos1970}(whether or not any polynomially bounded operator on a Hilbert space $H$ is similar
to a contraction).
J. Bourgain \cite{Bourgain1986} has shown that $R_{H}$ is similar to a contraction if $\phi^{'} \in BMOA$, A. Aleksandrov and V. Peller \cite{Aleksandrov1996} have shown that if $R_{H}$ is polynomially
 bounded then  $\phi^{'} \in BMOA.$ G. Pisier \cite{Pisier1997} and K. Davidson and
 V. Paulsen\cite{Davidson1997} give a negative answer to Halmos' problem via
 vector-Foguel-Hankel operators.
\item (Dual) Truncated Toeplitz operators: let$u$ is an inner function,
suppose $f\in L^{\infty}(\mathbb{T})$ and $H=\begin{pmatrix}
\begin{smallmatrix}
 f & u\bar{f} \\
uf & f
\end{smallmatrix}
\end{pmatrix}$, then
$R_{H}$ is unitary equivalent to the dual truncated Toeplitz operator $D_{f}$\cite{Sang2018,Sang2019,Sang2020}, furthermore, $R_{H}$  is equivalent after extension to truncated Toeplitz operator for invertible symbol \cite[Theorem 6.1]{Camara2020}.
\end{enumerate}

Given a closed unital subalgebra $A\subset L^{\infty}(\mathbb{T}),$
the $C^{*}-$algebra $\mathfrak{R}_{A}$ is defined by
\begin{align*}
\mathfrak{R}_{A}=\text{clos}\left\{ \sum^{n}_{i=1}\prod^{m}_{j=1}R_{H_{ij}}\bigg|
H_{ij}\in A_{2\times 2} \right\}.
\end{align*}
In fact, $\mathfrak{R}_{A}$ equals the $C^{*}-$algebra generated by
$\{R_{\alpha,\beta}\big|\alpha,\beta\in A\}$
and $\{R^{*}_{\phi,\psi}\big|\phi,\psi\in A\}.$
In this paper, we explore the structure of the $C^{*}-$algebra $\mathfrak{R}_{L^{\infty}(\mathbb{T})}.$

The earliest result on  the $C^{*}-$algebra $\mathfrak{R}_{PC(\mathbb{T})}$ due to
Gokhberg and Krupnik\cite{gokhberg1970algebra},
where $PC(\mathbb{T})$ denote the algebra of all piecewise continuous and left
continuous functions on $\mathbb{T}.$
They proved that the sequence
\begin{align*}
&0 \longrightarrow \mathfrak{C}(L^{2}(\mathbb{T})) \longrightarrow \mathfrak{R}_{PC(\mathbb{T})} \longrightarrow \mathscr{S} \longrightarrow 0.
\end{align*}
is exact.
The algebra $\mathscr{S}$ consist of matrix-valued functions of second order
$\mathrm{M}(\mathrm{t}, \mu)=(\alpha_{\mathrm{j} \mathrm{k}}(\mathrm{t}, \mu))_{j, \mathrm{k}}^{2}$ with the following properties:
\begin{itemize}
\item
$\alpha_{11}(t, \mu), \alpha_{22}(t, 1-\mu), \alpha_{12}(t, \mu), \alpha_{21}(t, \mu) \in C(\mathbb{T}\times [0,1]) $,

\item $\alpha_{12}(t, 0)=\alpha_{21}(t, 0)=\alpha_{12}(t, 1)=\alpha_{21}(t, 1)=0 \quad \forall t \in \mathbb{T}.$
\end{itemize}

This paper is organized as follows.
In section 2, we presents some preliminaries and basic properties of GISO.
In section 3 and section 4, we establish  the short exact sequences associated  of the $C^{*}-$algebras
generated by GISO with bounded symbols or quasicontinuous symbols, and obtain
the essential spectrum of GISO  and index forumla.
In section 5,
we establish vector
we obtain the necessary and sufficient conditions for invertibility and Fredholmness of GSIO
via equivalence after extension and Winer-Hopf factorization.
In the last section,
corresponding results apply for the spectrum of singular integral operators, Foguel-Hankel operators and dual truncated Toeplitz operators.
\section{Preliminaries}
The generalized singular
integral operator
 $R_{\begin{pmatrix}
\begin{smallmatrix}
 f & \phi \\
g & \psi
\end{smallmatrix}
\end{pmatrix}}$ can be expressed as an operator matrix with respect to
the decomposition
$L^{2}(\mathbb{T})=H^{2}\oplus \bar{z}\overline{H^{2}},$ the result
is of the form
\begin{gather}\label{m}
\begin{pmatrix}T_{f}& H^{*}_{\bar {\phi}}\\
		H_{g} &   \tilde{T}_{\psi} \end{pmatrix},
\end{gather}
where $T_{f}$ denote the Toeplitz operator on $H^{2}$ such that
\begin{align*}
T_{f}x=P_{+}(f x),\quad  x\in H^{2};
\end{align*}
$H_{g}$ denote the Hankel operator on $H^{2}$ such that
\begin{align*}
H_{g}x=P_{-}(g x),\quad  x\in H^{2};
\end{align*}
$H^{*}_{\bar {\phi}}$ denote the adjoint of Hankel operator  such that
\begin{align*}
H^{*}_{\bar {\beta}}y=P_{+}(\phi y),\quad  y\in \bar{z}\overline{H^{2}};
\end{align*}
$\tilde{T}_{\psi}$ denote the dual Toeplitz operator on $\bar{z}\overline{H^{2}}$ such that
\begin{align*}
 \tilde{T}_{\psi}y=P_{-}(\psi y),\quad  y\in \bar{z}\overline{H^{2}}.
\end{align*}
Converse, if an operator $T$ on $L^{2}(\mathbb{T})$ has form \eqref{m}, then $T$ is a GSIO.
Moreover, the generalized singular
integral operator
 $R_{\begin{pmatrix}
\begin{smallmatrix}
 f & \phi \\
g & \psi
\end{smallmatrix}
\end{pmatrix}}$  is unitarily equivalent to an operator matrix on $H_{2}^{2}.$
To illustrate this, we need to introduce two useful operators and their properties.
For $x\in L^{2},$ define
\begin{align*}
Vx(z)&=\bar{z}\overline{x(z)};\\
Jx(z)&=\bar{z}x(\bar{z}).
\end{align*}
Note that $V$ is an anti-unitary operator and $U$ is an unitary operator, and they have the following properties:
\begin{enumerate}
\item $\langle Vx,Vy\rangle=\langle y,x\rangle,$\quad $
\langle Ux,Uy\rangle=\langle x,y\rangle;$
\item $VM_{f}V=M_{\bar{f}},$\quad $UM_{f}U=M_{\tilde{f}},$ \quad \text{where} $\tilde{f}(z)=f(\bar{z});$
\item $VP_{-}=P_{+}V,$\quad$ UP_{-}=P_{+}U;$
\item $VH^{2}=\bar{z}\overline{H^{2}},$\quad $UH^{2}=\bar{z}\overline{H^{2}};$
\item $Uz^{n}=Vz^{n}=\bar{z}^{n+1}.$
\end{enumerate}
Using the operator $U,$ for $g\in L^{2},$ we can define the Hankel operator on $H^{2}$ by
\[\Gamma_{g}=UH_{g}.\]
The operator
$\begin{pmatrix}I& 0\\
		0 &   U \end{pmatrix}:
L^{2}=H^{2}\oplus \bar{z}\overline{H^{2}}\rightarrow H^{2}\oplus H^{2}$
is unitary.
A simple computation gives
\begin{align*}
&\begin{pmatrix}I& 0\\
		0 &   U \end{pmatrix}
\begin{pmatrix}T_{f}& H^{*}_{\bar {\phi}}\\
		H_{g} &   S_{\psi} \end{pmatrix}
\begin{pmatrix}I& 0\\
		0 &   U \end{pmatrix}\\
=&\begin{pmatrix}T_{f}& H^{*}_{\bar {\phi}}U\\
		UH_{g} &   US_{\psi}U \end{pmatrix}\\
=&\begin{pmatrix}T_{f}& \Gamma^{*}_{\bar {\phi}}\\
		\Gamma_{g} &   UP_{-}M_{\psi}P_{-}U \end{pmatrix}\\
=&\begin{pmatrix}T_{f}& \Gamma^{*}_{\bar {\phi}}\\
		\Gamma_{g} &   P_{+}UM_{\psi}UP_{+} \end{pmatrix}\\
=&\begin{pmatrix}T_{f}& \Gamma^{*}_{\bar {\phi}}\\
		\Gamma_{g} &   P_{+}M_{\tilde{\psi}}P_{+} \end{pmatrix}\\
=&\begin{pmatrix}T_{f}& \Gamma^{*}_{\bar {\phi}}\\
		\Gamma_{g} &   T_{\tilde{\psi}} \end{pmatrix}\\
=&\begin{pmatrix}T_{f}& \Gamma_{\tilde{\phi}}\\
		\Gamma_{g} &   T_{\tilde{\psi}} \end{pmatrix}.
\end{align*}
This shows that
the operator $R_{\begin{pmatrix}
\begin{smallmatrix}
 f & \phi \\
g & \psi
\end{smallmatrix}
\end{pmatrix}} :L^{2}\rightarrow L^{2}$ is unitary equivalent to
\begin{align*}
\begin{pmatrix}T_{f}& \Gamma_{\tilde{\phi}}\\
		\Gamma_{g} &   T_{\tilde{\psi}} \end{pmatrix}
:H^{2}\oplus H^{2}\rightarrow H^{2}\oplus H^{2}.
\end{align*}

Therefore, $R(\bar{z},0,\phi,\bar{z})$ is unitary equivalent to the Foguel-Hankel operator\cite{Bourgain1986}
\begin{align*}
\begin{pmatrix}T^{*}_{z}& \Gamma_{\tilde{\phi}}\\
		0 &   T_{z} \end{pmatrix}.
\end{align*}
\begin{example}
For $\alpha,\beta\in L^{\infty},$ the truncated singular integral operator
\begin{align*}
S^{u}_{\alpha,\beta}x=\alpha P_{u}x+\beta Q_{u}x,\quad x\in L^{2}.
\end{align*}
It can be write as an operator matrix
with respect to
the decomposition
$L^{2}(\mathbb{T})=H^{2}\oplus \bar{z}\overline{H^{2}},$
\begin{align*}
&\begin{pmatrix}T_{\alpha}+T_{(\beta-\alpha)u}T_{\bar{u}}& H^{*}_{\bar {\beta}}\\
		H_{\alpha}+H_{(\beta-\alpha)u}T_{\bar{u}} &   S_{\beta} \end{pmatrix}\\
=&\begin{pmatrix}T_{\alpha}& H^{*}_{\bar {\beta}}\\
		H_{\alpha} &   S_{\beta} \end{pmatrix}+
\begin{pmatrix}T_{(\beta-\alpha)u}T_{\bar{u}}& 0\\
		H_{(\beta-\alpha)u}T_{\bar{u}} &   0\end{pmatrix}\\
=&\begin{pmatrix}T_{\alpha}& H^{*}_{\bar {\beta}}\\
		H_{\alpha} &   S_{\beta} \end{pmatrix}+
\begin{pmatrix}T_{(\beta-\alpha)u}& 0\\
		H_{(\beta-\alpha)u} &   0\end{pmatrix}
\begin{pmatrix}T_{\bar{u}}& 0\\
		0 &   I\end{pmatrix}
\end{align*}
\end{example}
\begin{example}
Asymmetric dual truncated Toeplitz operator
$D^{\theta,\alpha}_{\phi}:(K^{2}_{\theta})^{\perp}\rightarrow (K^{2}_{\alpha})^{\perp}$
is unitarily equivalent to
some general singular
integral operator.
Let $h,g\in H^{2},$ we have
\begin{align*}
D^{\theta,\alpha}_{\phi}(\theta h+\bar{z}\bar{g})
&=(P_{-}+\alpha P_{+}\bar{\alpha})\phi (\theta h+\bar{z}\bar{g})\\
&=\alpha P_{+}\bar{\alpha}\phi\theta P_{+}h
+P_{-}\phi\theta P_{+}h
+\alpha P_{+}\bar{\alpha}\phi \bar{z}\bar{g}
+P_{-}\phi \bar{z}\bar{g}
\end{align*}
or
\begin{align*}
D^{\theta,\alpha}_{\phi}
\begin{pmatrix}M_\theta & 0\\
  0 &  I\end{pmatrix}
\binom{h}{\bar{z}\bar{g}}
&=\begin{pmatrix}M_\alpha & 0\\
  0 &  I\end{pmatrix}
\begin{pmatrix}T_{\bar{\alpha}\theta\phi} & H^{*}_{\bar{\phi}\alpha}\\
H_{\phi\theta} &  S_{\phi}\end{pmatrix}
\binom{h}{\bar{z}\bar{g}}
\end{align*}.
Hence
\begin{align*}
D^{\theta,\alpha}_{\phi}
&=\begin{pmatrix}M_\alpha & 0\\
  0 &  I\end{pmatrix}
\begin{pmatrix}T_{\bar{\alpha}\theta\phi} & H^{*}_{\bar{\phi}\alpha}\\
H_{\phi\theta} &  S_{\phi}\end{pmatrix}
\begin{pmatrix}M_{\bar{\theta}} & 0\\
  0 &  I\end{pmatrix},
\end{align*}
where $\begin{pmatrix}M_{\bar{\theta}} & 0\\
  0 &  I\end{pmatrix}:(K^{2}_{\theta})^{\perp}\rightarrow L^{2}$ and
$\begin{pmatrix}M_{\bar{\theta}} & 0\\
  0 &  I\end{pmatrix}: L^{2}\rightarrow(K^{2}_{\alpha})^{\perp} $
are unitary.
\end{example}

We begin our study of GSIO by considering some elementary properties.
\begin{proposition}
Let $H={\tiny\begin{pmatrix}
   f& \phi\\
   g & \psi \end{pmatrix}}\in L_{2\times2}^{2}(\mathbb{T}).$
\begin{enumerate}
\item  $R_{H}$ is bounded on $L^{2}(\mathbb{T})$ if and only if $f,\psi\in L^{\infty}$ and $g_{-},(\bar{\phi})_{-}\in \mathrm{BMO}(\mathbb{T}).$ Where $\mathrm{BMO}(\mathbb{T})=L^{\infty}(\mathbb{T})+\mathbb{H} L^{\infty}(\mathbb{T}).$

\item If $R_{H}$ is bounded,  then $R_{H}$ is zero if and only if $f=\psi=0$ and $g,\bar{\phi}\in H^{2}.$

\item If $R_{H}$ is bounded, then $R_{H}$ is compact if and only if $f=\psi=0$ and $g,\bar{\phi}\in H^{\infty}+C(\mathbb{T}).$

\item If $R_{H}$ is bounded, then $R(f,g,\phi,\psi)$ is
self-adjoint if and only if
$f$ and $\psi$ are real valued, and $g-\bar{\phi}\in H^{2}$.

\item If $R_{H}$ is bounded and positive, then
$f$ and $\psi$ are positive  and $g-\bar{\phi}\in H^{2}$.

\item If $R_{H}$ is bounded, then $R_{H}$ is
complex symmetric operator for $V$ if and only if
$f=\psi,$ where $Vf(z)=\bar{z}\bar{f}(z).$
\begin{proof}
(1)-(3) Clearly $R_{H}$ is  bounded (resp. zero, compact) if and only if
$T_{f},H^{*}_{\bar {\phi}},H_{g}$ and $S_{\psi}$ are bounded (resp. zero, compact).
Toeplitz operator $T_{a}$ is bounded \cite[7.8]{Douglas1998} (resp., zero, compact\cite[p.94]{Brown1963/64}) if and only if its symbol $a$ is bounded(resp., zero, zero), Hankel operator $H_{a}$ is bounded\cite[Theorem 1.3]{Peller2003}(resp., zero, compact \cite[Theorem 5.5]{Peller2003})
if and only if $a_{-}\in BMO$ (resp.,$a\in H^{2},a\in H^{\infty}+C(\mathbb{T})$),
the conclusion follows.

(4) By the matrix represention \eqref{m},
we have $R_{H}$ is self-adjoint if and only if
\begin{align*}
\begin{pmatrix}T_{f}& H^{*}_{\bar {\phi}}\\
		H_{g} &   S_{\psi} \end{pmatrix}
=\begin{pmatrix}T_{\bar{f}}& H^{*}_{g}\\
		H_{\bar {\phi}} &   S_{\bar{\psi}} \end{pmatrix}
\end{align*}
if and only if
$T_{f}=T_{\bar{f}},H_{g}=H_{\bar{\phi}}$ and $S_{\psi}=S_{\bar{\psi}}.$
$T_{f}=T_{\bar{f}}$ is equivalent to $f$ is real,
$H_{g}=H_{\bar{\phi}}$ is equivalent to $g-\bar{\phi}\in H^{2}$,
$S_{\psi}=S_{\bar{\psi}}$ is equivalent to $\psi$ is real.

(4) If $R_{H}$ is
positive, then
\begin{equation}\label{ber}
\begin{split}
0 & \leq\langle R_{H}k_{z},k_{z}\rangle\\
&=\langle (P_{+}fP_{+}+P_{-}g P_{+}
+P_{+}\phi P_{-}+P_{-}\psi P_{-})k_{z},k_{z}\rangle\\
&=\langle (P_{+}f P_{+}+P_{-}g P_{+}
+P_{+}\phi P_{-}+P_{-}\psi P_{-})P_{+}k_{z},P_{+}k_{z}\rangle\\
&=\langle P_{+}(P_{+}f P_{+}+P_{-}g P_{+}
+P_{+}\phi P_{-}+P_{-}\psi P_{-})P_{+}k_{z},k_{z}\rangle\\
&=\langle P_{+}f P_{+}k_{z},k_{z}\rangle\\
&=\langle f k_{z},k_{z}\rangle\\
&=\int^{2\pi}_{0}f(e^{i\theta})|k_{z}(e^{i\theta})|^{2}\frac{d\theta}{2\pi},
\end{split}
\end{equation}
where $k_{z}(\omega)=\frac{\sqrt{1-|z|^{2}}}{1-\bar{z}\omega}$ is
the normalized reproducing kernel of $H^{2}.$
The last equality is the Poisson integral of $f$,
so $f$ is positive almost everywhere on $\mathbb{T}.$
Similarly,
\begin{align*}
0 & \leq\langle R(f,g,\phi,\psi)\bar{z}\bar{k}_{z},\bar{z}\bar{k}_{z}\rangle\\
& =\langle (P_{+}fP_{+}+P_{-}g P_{+}
+P_{+}\phi P_{-}+P_{-}\psi P_{-})P_{-}\bar{z}\bar{k}_{z},P_{-}\bar{z}\bar{k}_{z}\rangle\\
& =\langle P_{-}\psi P_{-}\bar{z}\bar{k}_{z},P_{-}\bar{z}\bar{k}_{z}\rangle\\
& =\langle \psi k_{z},k_{z}\rangle\\
&=\int^{2\pi}_{0}\psi(e^{i\theta})|k_{z}(e^{i\theta})|^{2}\frac{d\theta}{2\pi},
\end{align*}
so $\psi$ is positive almost everywhere on $\mathbb{T}.$
Since positive opertor is self-adjoint and (4),$g-\bar{\phi}\in H^{2}.$

(5)
By the definition of complex symmetric operator \cite{Garcia2006}, we have
$R_{H}$ is complex symmetric with the conjugation $V$ if and only if
$VR_{H}V=R^{*}_{H}.$
Using the properties of $V$ yields
\begin{equation}\label{VRV}
\begin{split}
&VR_{H}V\\
=&V(P_{+}fP_{+}+P_{-}g P_{+}
+P_{+}\phi P_{-}+P_{-}{\psi}P_{-})V\\
=&VP_{+}fP_{+}V+VP_{-}gP_{+}V
+VP_{+}\phi P_{-}V+VP_{-}{\psi}P_{-}V\\
=&P_{-}VfVP_{-}+P_{+}VgVP_{-}
+P_{-}V\phi VP_{+}+P_{+}V{\psi}VP_{+}\\
=&P_{-}\bar{f}P_{-}+P_{+}\bar{g}P_{-}
+P_{-}\bar{\phi}P_{+}+P_{+}\bar{\psi}P_{+}\\
=&\begin{pmatrix}T_{\bar{\psi}}& H^{*}_{g}\\
H_{\bar{\phi}} &   \tilde{T}_{\bar{f}} \end{pmatrix}.
\end{split}
\end{equation}
On the other hand,
$R^{*}_{H}
=\begin{pmatrix}T_{\bar{f}}& H^{*}_{g}\\
		H_{\bar {\phi}} &   S_{\bar{\psi}} \end{pmatrix}.$
It follows that
$VR_{H}V=R^{*}_{H}$ holds  if and only if
$T_{\bar{f}}=T_{\bar{\psi}}$ and $S_{\bar{f}}=S_{\bar{\psi}}$ hold
if and only if
$f=\psi.$
\end{proof}
\end{enumerate}
\end{proposition}
\section{$C^{*}-$algebras $\mathfrak{R}_{L^{\infty}}$}
Recall the $C^{*}-$algebra $\mathfrak{R}_{L^{\infty}}$ is defined by
\begin{align*}
\mathfrak{R}_{L^{\infty}}=\text{clos}\left\{ \sum^{n}_{i=1}\prod^{m}_{j=1}R_{H_{ij}}\bigg|
H_{ij}\in L^{\infty}_{2\times 2}(\mathbb{T}) \right\}.
\end{align*}

Let $\mathfrak{SR}_{L^{\infty}}$ be the closed ideal of
$\mathfrak{R}_{L^{\infty}}$ generated by operators of
the form
\begin{align}\label{semi}
R_{\begin{pmatrix}
\begin{smallmatrix}
 f_{1} & \phi_{1} \\
g_{1} & \psi_{1}
\end{smallmatrix}
\end{pmatrix}}
R_{\begin{pmatrix}
\begin{smallmatrix}
 f_{2} & \phi_{2} \\
g_{2} & \psi_{2}
\end{smallmatrix}
\end{pmatrix}}-
R_{\begin{pmatrix}\begin{smallmatrix}
 f_{1}f_{2} & \phi \\
g & \psi_{1}\psi_{2}
\end{smallmatrix}
\end{pmatrix}}
\end{align}
where$f_{i},g_{i},\phi_{i},\psi_{i},
g,\phi$ are in $L^{\infty}(\mathbb{T})(i=1,2).$
Furthermore, the $C^{*}-$algebra $\mathfrak{R}_{L^{\infty}}$ equals the algebra generated by Riesz projection and all multiplication operators with $L^{\infty}(\mathbb{T})$ symbols, i.e.
\begin{align*}
\mathfrak{R}_{L^{\infty}}=\text{clos span}\left\{P,M_{\phi}|
\phi\in L^{\infty}(\mathbb{T}) \right\}.
\end{align*}
Next,we will establish the symbol map of $\mathfrak{R}_{L^{\infty}}$ with the normalized reproducing kernel of $H^{2}.$
\begin{lemma}\label{keylem}
Let $H_{i}=\begin{pmatrix}
\begin{smallmatrix}
 f_{i} & \phi_{i} \\
g_{i} & \psi_{i}
\end{smallmatrix}
\end{pmatrix}\in \bigcap_{p\geq1}L_{2\times 2}^{p}(\mathbb{T}),i\in \mathbb{Z}_{+}.$
\begin{enumerate}
\item The radial limit
\begin{align*}
\lim_{r\rightarrow 1^{-}}\big\langle R_{H_{1}}\cdots R_{H_{m}}k_{r\xi},k_{r\xi}\big\rangle= f_{1}(\xi)\cdots f_{m}(\xi)\quad a.e. ~on ~\mathbb{T}.
\end{align*}

\item The radial limit
\begin{align*}
\lim_{r\rightarrow 1^{-}}\big\langle R_{H_{1}}\cdots R_{H_{m}}\bar{z}\bar{k}_{r\xi},\bar{z}\bar{k}_{r\xi}\big\rangle= \psi_{1}(\xi)\cdots \psi_{m}(\xi) \quad a.e. ~on~  \mathbb{T}.
\end{align*}

\item If $g,\phi\in \cap_{p\geq1}L^{p}(\mathbb{T}),$ then
$\prod_{i=1}^{n}R_{H_i}-
R_{\begin{pmatrix}\begin{smallmatrix}
 \Pi_{i=1}^{n}f_{i} & \phi \\
g & \Pi_{i=1}^{n}\psi_{i}
\end{smallmatrix}
\end{pmatrix}}
\in \mathfrak{SR}_{L^{\infty}}.$
		
\item If $T\in \mathfrak{SR}_{L^{\infty}},$ then
		\begin{align*}
			\lim_{r\rightarrow 1}\langle Tk_{r\xi},k_{r\xi}\rangle & =0,\\
           \lim_{r\rightarrow 1}\langle T\bar{z}\bar{k}_{r\xi},\bar{z}\bar{k}_{r\xi}\rangle & =0.
		\end{align*}
\item  The uniform limit of {\rm GSIO} is also a {\rm GSIO}.
\end{enumerate}
\end{lemma}

\begin{proof}
	(1)
	We will prove this lemma by induction on $m.$
	For $m=1$, applying \eqref{ber},  we obtain
\begin{align*}
\langle R_{H_{1}} k_{r\xi},k_{r\xi}\rangle
=\int^{2\pi}_{0}f_{1}(e^{i\theta})|k_{r\xi}(e^{i\theta})|^{2}\frac{d\theta}{2\pi}
\end{align*}
where $|k_{r\xi}|^{2}$ is the Poisson kernel for $r\xi \in\mathbb{D}.$
By Fatou's theorem,
\begin{align*}
\lim_{r\rightarrow 1}\langle R_{H_{1}} k_{r\xi},k_{r\xi}\rangle=f_{1}(\xi)
\end{align*}
for almost all $\xi \in\mathbb{T}.$

Let $m\geq2,$ assume the result true up to $n-1.$
A simple computation gives
\begin{align*}
&\langle R_{H_{1}}R_{H_{2}}\cdots R_{H_{n}}k_{r\xi},k_{r\xi}\rangle
=\langle R_{H_2}\cdots R_{H_n}k_{r\xi},R^{*}_{H_1}k_{r\xi}\rangle\\
=&\langle R_{H_2}\cdots R_{n}k_{r\xi},(P_{+}\bar{f_{1}}P_{+}+P_{+}\bar{g}_{1}P_{-}
+P_{-}\bar{\phi}_{1}P_{+}+P_{-}\bar{\psi_{1}}P_{-})k_{r\xi}\rangle\\
=&\langle R_{H_2}\cdots R_{H_n}k_{r\xi},(P_{+}\bar{f_{1}}P_{+}
+P_{-}\bar{\phi}_{1}P_{+})k_{r\xi}\rangle\\
=&\langle R_{H_2}\cdots R_{H_n}k_{r\xi},P_{+}\bar{f_{1}}P_{+}k_{r\xi}\rangle
+\langle R_{H_2}\cdots R_{H_n}k_{r\xi},P_{-}\bar{\phi}_{1}P_{+}k_{r\xi}\rangle\\
=&\langle R_{H_2}\cdots R_{H_n}k_{r\xi},P_{+}\bar{f_{1}}k_{r\xi}\rangle
+\langle R_{H_2}\cdots R_{H_n}k_{r\xi},P_{-}\bar{\phi}_{1}k_{r\xi}\rangle\\
=&\langle R_{2}\cdots R_{H_n}k_{r\xi},P_{+}(\bar{f}_{1+}+\bar{f}_{1-})k_{r\xi}\rangle
+\langle R_{H_2}\cdots R_{H_n}k_{r\xi},P_{-}\bar{\phi}_{1}k_{r\xi}\rangle\\
=&\langle R_{H_2}\cdots R_{H_n}k_{r\xi},\bar{f}_{1+}(r\xi)k_{r\xi}+P_{+}\bar{f}_{1-}k_{r\xi}\rangle
+\langle R_{H_2}\cdots R_{H_n}k_{r\xi},P_{-}\bar{\phi}_{1}k_{r\xi}\rangle\\
=&f_{1+}(r\xi)\langle R_{H_2}\cdots R_{H_n}k_{r\xi},k_{r\xi}\rangle
+\langle R_{H_2}\cdots R_{H_n}k_{r\xi},P_{+}\bar{f}_{1-}k_{r\xi}\rangle\\
&+\langle R_{H_2}\cdots R_{H_n}k_{r\xi},P_{-}\bar{\phi}_{1}k_{r\xi}\rangle,
\end{align*}
where $f_{1+}=P_{+}f_{1},f_{1-}=P_{-}f_{1}.$
Note that
\begin{align*}
&\langle R_{H_2}\cdots R_{H_n}k_{r\xi},P_{+}\bar{f}_{1-}k_{r\xi}\rangle
=\langle f_{1-}P_{+}R_{H_2}\cdots R_{H_n}k_{r\xi},k_{r\xi}\rangle\\
=&\langle f_{1-}P_{+}R_{H_2}\cdots R_{H_n}k_{r\xi},P_{+}k_{r\xi}\rangle
=\langle P_{+}f_{1-}P_{+}R_{H_2}\cdots R_{H_n}k_{r\xi},k_{r\xi}\rangle\\
=&\langle P_{+}f_{1-}P_{+}(P_{+}f_{2}P_{+}+P_{-}g_{2}P_{+}
+P_{+}{\phi}_{2}P_{-}+P_{-}\psi_{2}P_{-})R_{3}\cdots R_{n}k_{r\xi},k_{r\xi}\rangle\\
=&\langle (P_{+}f_{1-}P_{+}f_{2}P_{+}
+P_{+}f_{1-}P_{+}{\phi}_{2}P_{-})R_{H_3}\cdots R_{H_n}k_{r\xi},k_{r\xi}\rangle\\
=&\langle (P_{+}f_{1-}f_{2}P_{+}
+P_{+}f_{1-}{\phi}_{2}P_{-})R_{H_3}\cdots R_{H_n}k_{r\xi},k_{r\xi}\rangle\\
=&\langle R_{\begin{pmatrix}
\begin{smallmatrix}
 f_{2}f_{1-} & {\phi}_{2}f_{1-} \\
0 & 0
\end{smallmatrix}
\end{pmatrix}}R_{H_3}\cdots R_{H_n}k_{r\xi},k_{r\xi}\rangle,
\end{align*}
\begin{align*}
|\langle R_{H_2}\cdots R_{H_n}k_{r\xi},P_{-}\bar{\phi}_{1}k_{r\xi}\rangle|
&\leq\|R_{H_2}\cdots R_{H_n}\|\|k_{r\xi}\|\|P_{-}\bar{\phi}_{1}k_{r\xi}\|
\end{align*}
and
\begin{align*}
&\|P_{-}\bar{\phi}_{1}k_{r\xi}\|\\
=&\|P_{-}(\bar{\phi}_{1+}+\bar{\phi}_{1-})k_{r\xi}\|\\
=&\|P_{-}\bar{\phi}_{1+}k_{r\xi}\|\\
=&\|(I-P_{+})(\bar{\phi}_{1+}k_{r\xi})\|\\
=&
\bigg(\int^{2\pi}_{0}|\bar{\phi}_{1+}(e^{i\theta})
-\bar{\phi}_{1+}(r\xi)|^{2}|k_{r\xi}(e^{i\theta})|^{2}\frac{d\theta}{2\pi}\bigg)^{\frac{1}{2}}
\rightarrow 0, a.e.(r\rightarrow 1^{-}).
\end{align*}
By induction hypothesis, the result holds.

(2)Using the properties of $V,$ we have
\begin{align*}
\langle R_{H_1}\cdots R_{H_m}\bar{z}\bar{k}_{r\xi},\bar{z}\bar{k}_{r\xi}\rangle
=&\langle R_{H_1}\cdots R_{H_m}Vk_{r\xi},Vk_{r\xi}\rangle\\
=&\langle VVk_{r\xi},VR_{H_1}\cdots R_{H_m}Vk_{r\xi}\rangle\\
=&\langle k_{r\xi},VR_{H_1}\cdots R_{H_m}Vk_{r\xi}\rangle.\\
=&\overline{\langle (VR_{H_1}V)\cdots (VR_{H_m}V)k_{r\xi},k_{r\xi}\rangle.}
\end{align*}
By \eqref{VRV}, we have
\begin{align*}
VR_{H_i}V=
R_{\begin{pmatrix}
\begin{smallmatrix}
 \bar{\psi_{i}} & \bar{g_{i}} \\
\bar{\phi_{i}} & \bar{f_{i}}
\end{smallmatrix}
\end{pmatrix}}
\quad  1\leq i \leq m.
\end{align*}
Hence Lemma \ref{keylem} (1) implies the result.

(3)
For $k=2$, by the definition \ref{semi},we have
\begin{align*}
R_{\begin{pmatrix}
\begin{smallmatrix}
 f_{1} & \phi_{1} \\
g_{1} & \psi_{1}
\end{smallmatrix}
\end{pmatrix}}
R_{\begin{pmatrix}
\begin{smallmatrix}
 f_{2} & \phi_{2} \\
g_{2} & \psi_{2}
\end{smallmatrix}
\end{pmatrix}}-
R_{\begin{pmatrix}\begin{smallmatrix}
 f_{1}f_{2} & \phi \\
g & \psi_{1}\psi_{2}
\end{smallmatrix}
\end{pmatrix}}\in \mathfrak{SR}_{L^{\infty}}.
\end{align*}
Assume the result true up to $n-1.$ Observe that
\begin{align*}
&\prod_{i=1}^{n}R_{\begin{pmatrix}
\begin{smallmatrix}
 f_{i} & \phi_{i} \\
g_{i} & \psi_{i}
\end{smallmatrix}
\end{pmatrix}}-R_{\begin{pmatrix}\begin{smallmatrix}
 \Pi_{i=1}^{n}f_{i} & \phi \\
g & \Pi_{i=1}^{n}\psi_{i}
\end{smallmatrix}
\end{pmatrix}}\\
=&\prod_{i=1}^{n}R_{\begin{pmatrix}
\begin{smallmatrix}
 f_{i} & \phi_{i} \\
g_{i} & \psi_{i}
\end{smallmatrix}
\end{pmatrix}}
-R_{\begin{pmatrix}
\begin{smallmatrix}
 f_{1} & \phi_{1} \\
g_{1} & \psi_{1}
\end{smallmatrix}
\end{pmatrix}}
R_{\begin{pmatrix}\begin{smallmatrix}
 \Pi_{i=2}^{n}f_{i} & \phi \\
g & \Pi_{i=2}^{n}\psi_{i}
\end{smallmatrix}
\end{pmatrix}}\\
&\quad\quad\quad\quad\quad+R_{\begin{pmatrix}
\begin{smallmatrix}
 f_{1} & \phi_{1} \\
g_{1} & \psi_{1}
\end{smallmatrix}
\end{pmatrix}}
R_{\begin{pmatrix}\begin{smallmatrix}
 \Pi_{i=1}^{n}f_{i} & \phi \\
g & \Pi_{i=1}^{n}\psi_{i}
\end{smallmatrix}
\end{pmatrix}}
-R_{\begin{pmatrix}\begin{smallmatrix}
 \Pi_{i=1}^{n}f_{i} & \phi \\
g & \Pi_{i=1}^{n}\psi_{i}
\end{smallmatrix}
\end{pmatrix}}\\
=&R_{\begin{pmatrix}
\begin{smallmatrix}
 f_{1} & \phi_{1} \\
g_{1} & \psi_{1}
\end{smallmatrix}
\end{pmatrix}}\bigg(\underbrace{\prod_{i=2}^{n}R_{\begin{pmatrix}
\begin{smallmatrix}
 f_{i} & \phi_{i} \\
g_{i} & \psi_{i}
\end{smallmatrix}
\end{pmatrix}}-R_{\begin{pmatrix}\begin{smallmatrix}
 \Pi_{i=2}^{n}f_{i} & \phi \\
g & \Pi_{i=2}^{n}\psi_{i}
\end{smallmatrix}
\end{pmatrix}}}_{\in\mathfrak{SR}_{L^{\infty}}}\bigg)\\
&\quad\quad\quad\quad\quad+\underbrace{R_{\begin{pmatrix}
\begin{smallmatrix}
 f_{1} & \phi_{1} \\
g_{1} & \psi_{1}
\end{smallmatrix}
\end{pmatrix}}
R_{\begin{pmatrix}\begin{smallmatrix}
 \Pi_{i=2}^{n}f_{i} & \phi \\
g & \Pi_{i=2}^{n}\psi_{i}
\end{smallmatrix}
\end{pmatrix}}
-R_{\begin{pmatrix}\begin{smallmatrix}
 \Pi_{i=1}^{n}f_{i} & \phi \\
g & \Pi_{i=1}^{n}\psi_{i}
\end{smallmatrix}
\end{pmatrix}}}_{\in\mathfrak{SR}_{L^{\infty}}}.
\end{align*}
By induction hypothesis,the result holds.
	
(4)Suppose $g,\phi \in L^{\infty}.$
Linear combinations of operators of the form
\begin{align*}
&R_{H_1}R_{H_2}\cdots R_{H_n-1}\bigg(R_{H_n}R_{H_n+1}
-R_{\begin{pmatrix}\begin{smallmatrix}
 f_{n}f_{n+1} & \phi \\
g & \psi_{n}\psi_{n+1}
\end{smallmatrix}
\end{pmatrix}}\bigg)
R_{H_{n+2}}R_{H_{n+3}}\cdots R_{H_{n+k}}\\
=&R_{H_1}R_{H_2}\cdots R_{H_{n-1}}R_{H_n}R_{H_{n+1}}R_{H_{n+2}}R_{H_{n+3}}\cdots R_{H_{n+k}}\\
&-R_{H_1}R_{H_2}\cdots R_{H_{n-1}}
R_{\begin{pmatrix}\begin{smallmatrix}
 f_{n}f_{n+1} & \phi \\
g & \psi_{n}\psi_{n+1}
\end{smallmatrix}
\end{pmatrix}}
R_{H_{n+2}}R_{H_{n+3}}\cdots R_{H_{n+k}},
\end{align*}
form a dense subset of $\mathfrak{SR}_{L^{\infty}}.$
Lemma \ref{keylem} (1)(2) gives the result.
	
(5) If $R$ is a bounded operator on $L^{2}$ and $\lim_{n\rightarrow \infty}||R_{H_n}-R||=0,$ then
\begin{align*}
&\lim_{n\rightarrow\infty}||P_{+}(R_{H_n}-R)P_{+}||\leq\lim_{n\rightarrow \infty}||R_{H_n}-R||=0,\\
&\lim_{n\rightarrow\infty}||P_{-}(R_{H_n}-R)P_{+}||\leq\lim_{n\rightarrow \infty}||R_{H_n}-R||=0,\\
&\lim_{n\rightarrow\infty}||P_{+}(R_{H_n}-R)P_{-}||\leq\lim_{n\rightarrow \infty}||R_{H_n}-R||=0,\\
&\lim_{n\rightarrow\infty}||P_{-}(R_{H_n}-R)P_{-}||\leq\lim_{n\rightarrow \infty}||R_{H_n}-R||=0.
\end{align*}
Since
\begin{align*}
P_{+}R_{H_n}P_{+}|_{H^{2}}&=T_{f_{n}},\\
P_{-}R_{H_n}P_{+}|_{H^{2}}&=H_{g_{n}},\\
P_{+}R_{H_n}P_{-}|_{\bar{z}\overline{H^{2}}}&=H^{*}_{\bar{\varphi}_{n}},\\
P_{-}R_{H_n}P_{-}|_{\bar{z}\overline{H^{2}}}&=\tilde{T}_{\psi_{n}},
\end{align*}
and
\begin{align*}
&\|T_{\bar{z}}P_{+}RP_{+}T_{z}-P_{+}RP_{+}\|\\
=&\|T_{\bar{z}}P_{+}RP_{+}T_{z}-T_{\bar{z}}T_{f_{n}}T_{z}
+T_{f_{n}}-P_{+}RP_{+}\|\\
\leq &\|T_{\bar{z}}P_{+}RP_{+}T_{z}-T_{\bar{z}}T_{f_{n}}T_{z}\|
+\|T_{f_{n}}-P_{+}RP_{+}\|\\
\leq &\|T_{\bar{z}}(P_{+}RP_{+}-T_{f_{n}})T_{z}\|
+\|T_{f_{n}}-P_{+}RP_{+}\|\\
\leq &\|T_{\bar{z}}\|\|P_{+}RP_{+}-T_{f_{n}}\|\|T_{z}\|
+\|T_{f_{n}}-P_{+}RP_{+}\|\rightarrow 0\quad (n\rightarrow \infty).
\end{align*}
it follows that $T_{\bar{z}}P_{+}RP_{+}T_{z}=P_{+}RP_{+}.$
We have $P_{+}RP_{+}|_{H^{2}}$ is a Toeplitz operator, because  an operator $T$ is a Toeplitz operator if and only if
$T_{\bar{z}}TT_{z}=T$ \cite[Theorem 6]{Brown1963/64}.
Moreover,
\begin{align*}
&\|P_{-}RP_{+}T_{z}-S_{z}P_{+}RP_{+}\|\\
= &\|P_{-}RP_{+}T_{z}-P_{-}R_{H_n}P_{+}T_{z}
+S_{z}P_{-}R_{H_n}P_{+}-S_{z}P_{+}RP_{+}\|\\
\leq &\|P_{-}RP_{+}T_{z}-P_{-}R_{H_n}P_{+}T_{z}\|
+\|S_{z}P_{-}R_{H_n}P_{+}-S_{z}P_{+}RP_{+}\|\\
\leq &\|P_{-}RP_{+}-P_{-}R_{H_n}P_{+}\|\|T_{z}\|
+\|S_{z}\|\|P_{-}R_{n}P_{+}-P_{+}RP_{+}\|
\end{align*}
shows that $P_{-}RP_{+}T_{z}=S_{z}P_{+}RP_{+}.$
Since an operator $H$ is a Hankel operator if and only if
$HT_{z}=\tilde{T}_{z}H$ \cite[Theorem 1.8]{Peller2003},
we have $P_{-}RP_{+}|_{H^{2}}$ is a Hankel operator.
Similary, $P_{+}RP_{-}|_{\bar{z}\overline{H^{2}}}$ is the adjoint of  a Hankel operator.
By $VT_{\psi}V=\tilde{T}_{\bar{\psi}}$, then $P_{-}R_{H_n}P_{-}|_{\bar{z}\overline{H^{2}}}$ is a dual Toeplitz operator. Hence $R$ is a GSIO.
\end{proof}

\begin{theorem}\label{main}
The sequence
\begin{align*}
0 \longrightarrow \mathfrak{SR}_{L^{\infty}} \longrightarrow \mathfrak{R}_{L^{\infty}}
\longrightarrow  L_{2}^{\infty}(\mathbb{T})\longrightarrow 0
\end{align*}
is a short exact sequence; that is, the quotient algebra
$\mathfrak{R}_{L^{\infty}}/\mathfrak{SR}_{L^{\infty}}$
is *-isometrically isomorphic to $L^{\infty} \oplus L^{\infty}.$
\end{theorem}
\begin{proof}
Linear combinations of operators of the form $\prod^{m}_{j=1}R_{\begin{pmatrix}
\begin{smallmatrix}
 f_{i} & \phi_{i} \\
g_{i} & \psi_{i}
\end{smallmatrix}
\end{pmatrix}}$
span a dense subset of $\mathfrak{R}_{L^{\infty}},$ compute
\begin{align*}
\prod^{m}_{j=1}R_{\begin{pmatrix}
\begin{smallmatrix}
 f_{i} & \phi_{i} \\
g_{i} & \psi_{i}
\end{smallmatrix}
\end{pmatrix}}
=R_{\begin{pmatrix}\begin{smallmatrix}
 \Pi_{i=1}^{m}f_{i} & 0 \\
0 & \Pi_{i=1}^{m}\psi_{i}
\end{smallmatrix}
\end{pmatrix}}+
\underbrace{\prod^{m}_{j=1}R_{\begin{pmatrix}
\begin{smallmatrix}
 f_{i} & \phi_{i} \\
g_{i} & \psi_{i}
\end{smallmatrix}
\end{pmatrix}}
-R_{\begin{pmatrix}\begin{smallmatrix}
 \Pi_{i=1}^{m}f_{i} & 0 \\
0 & \Pi_{i=1}^{m}\psi_{i}
\end{smallmatrix}
\end{pmatrix}}.}_{\in \mathfrak{SR}_{L^{\infty}}(By Lemma \ref{keylem}(3))}
\end{align*}
This shows that operators of the form
\[T=R_{\begin{pmatrix}
\begin{smallmatrix}
 f & 0 \\
0 & \psi
\end{smallmatrix}
\end{pmatrix}}+E_{0},\quad f,\psi\in L^{\infty}, E_{0}\in\mathfrak{SR}_{L^{\infty}}.\]
form a dense subset of $\mathfrak{R}_{L^{\infty}}.$
Therefore, for every operator $T$ in $\mathfrak{R}_{L^{\infty}},$
there exists a sequence of operators
\[T_{n}=R_{\begin{pmatrix}
\begin{smallmatrix}
 f_{n} & 0 \\
0 & \psi_{n}
\end{smallmatrix}
\end{pmatrix}}+E_{n},\quad E_{n}\in \mathfrak{SR}_{L^{\infty}}\] such that
$\lim_{n\rightarrow\infty}||T_{n}-T||=0.$
By Lemma \ref{keylem}(1)and(4), we have
\begin{align*}
f_{n}(\xi)=&\lim_{r\rightarrow 1^{-}}\langle T_{n}k_{r\xi},k_{r\xi}\rangle.
\end{align*}
and
\begin{align}\label{inf}
|f_{n}(\xi)-f_{m}(\xi)|\leq \|T_{n}-T_{m}\|.
\end{align}
So $\{f_{n}(\xi)\}$ is  a Cauchy sequence.
Define
\begin{align*}
f(\xi)\triangleq\lim_{n\rightarrow \infty}f_{n}(\xi).
\end{align*}
we then have
\begin{align*}
&\mid\lim_{r\rightarrow 1^{-}}\langle Tk_{r\xi},k_{r\xi}\rangle-f(\xi)\mid\\
=&|\lim_{r\rightarrow 1^{-}}\langle Tk_{r\xi},k_{r\xi}\rangle-\lim_{r\rightarrow 1^{-}}\langle T_{n}k_{r\xi},k_{r\xi}\rangle+\lim_{r\rightarrow 1^{-}}\langle T_{n}k_{r\xi},k_{r\xi}\rangle-f_{n}(\xi)+f_{n}(\xi)-f(\xi)|\\
\leq &\mid\lim_{r\rightarrow 1^{-}}\langle Tk_{r\xi},k_{r\xi}\rangle-\lim_{r\rightarrow 1^{-}}\langle T_{n}k_{r\xi},k_{r\xi}\rangle\mid+\mid f_{n}(\xi)-f(\xi) \mid\\
\leq &\|T-T_{n}\|+\mid f_{n}(\xi)-f(\xi) \mid
\end{align*}
and it follows that
\begin{align*}
\lim_{r\rightarrow 1^{-}}\langle Tk_{r\xi},k_{r\xi}\rangle=f(\xi).
\end{align*}
Similarly, define
\begin{align*}
\psi(\xi)\triangleq\lim_{n\rightarrow \infty}\psi_{n}(\xi),
\end{align*}
we have
\begin{align*}
\lim_{r\rightarrow 1^{-}}\langle T\bar{z}\bar{k}_{r\xi},\bar{z}\bar{k}_{r\xi}\rangle=\psi(\xi).
\end{align*}
Using\eqref{inf}, $\lim_{n\rightarrow\infty}\|f_{n}-f\|_{\infty}=0.$ Similarly,
$\lim_{n\rightarrow\infty}\|\psi_{n}-\psi\|_{\infty}=0.$
Thus
\begin{align*}
\|R_{\begin{pmatrix}
\begin{smallmatrix}
 f_{n}-f & 0 \\
0 & \psi_{n}-\psi
\end{smallmatrix}
\end{pmatrix}}\|
\leq \|f_{n}-f\|+\|\psi_{n}-\psi\|\rightarrow 0 \quad (n\rightarrow\infty).
\end{align*}
Let $E=T-R_{\begin{pmatrix}
\begin{smallmatrix}
 f & 0 \\
0 & \psi
\end{smallmatrix}
\end{pmatrix}},$
we have $\lim_{n\rightarrow\infty}\|E_{n}-E\|=0,$
since $\mathfrak{SR}_{L^{\infty}}$ is closed, $E\in\mathfrak{SR}_{L^{\infty}}.$
It follows that
$T$ have the following form
\begin{align*}
T=R_{\begin{pmatrix}
\begin{smallmatrix}
 f & 0 \\
0 & \psi
\end{smallmatrix}
\end{pmatrix}}+E,\quad f,\psi\in L^{\infty}(\mathbb{T}), E\in\mathfrak{SR}_{L^{\infty}}.
\end{align*}

Define the map $\rho:\mathfrak{R}_{L^{\infty}}\rightarrow L_{2}^{\infty}(\mathbb{T})$ by
\begin{align}\label{sym}
\rho(T)(\xi)=\left(\lim_{r\rightarrow 1^{-}}\langle Tk_{r\xi},k_{r\xi}\rangle,\lim_{r\rightarrow 1^{-}}\langle T\bar{z}\bar{k}_{r\xi},\bar{z}\bar{k}_{r\xi}\rangle\right).
\end{align}
Recall the norm of $L_{2}^{\infty}(\mathbb{T}),$
$\|(a,b)\|=max\{\|a\|_{\infty},\|b\|_{\infty}\}.$
Clearly, $\|\rho(T)\|\leq\|T\|.$
The map $\rho$ is linear,contractive, and preserves conjugation.
Moreover, \[\rho(T)=(f,\psi).\]
If $A_{1},A_{2}\in \mathfrak{R}_{L^{\infty}},$
and
\begin{align*}
A_{1}=R_{\begin{pmatrix}
\begin{smallmatrix}
 f_1 & 0 \\
0 & \psi_1
\end{smallmatrix}
\end{pmatrix}}+E_{1},\quad
A_{2}=&R_{\begin{pmatrix}
\begin{smallmatrix}
 f_2 & 0 \\
0 & \psi_2
\end{smallmatrix}
\end{pmatrix}}+E_{2},\quad
E_{1},E_{2}\in\mathfrak{SR}_{L^{\infty}},
\end{align*}
then
\begin{align*}
A_{1}A_{2}=R_{\begin{pmatrix}
\begin{smallmatrix}
 f_1 & 0 \\
0 & \psi_1
\end{smallmatrix}
\end{pmatrix}}
R_{\begin{pmatrix}
\begin{smallmatrix}
 f_2 & 0 \\
0 & \psi_2
\end{smallmatrix}
\end{pmatrix}}
+\underbrace{R_{\begin{pmatrix}
\begin{smallmatrix}
 f_1 & 0 \\
0 & \psi_1
\end{smallmatrix}
\end{pmatrix}}E_{2}
+E_{1}R_{\begin{pmatrix}
\begin{smallmatrix}
 f_2 & 0 \\
0 & \psi_2
\end{smallmatrix}
\end{pmatrix}}
+E_{1}E_{2}}_{\in \mathfrak{SR}_{L^{\infty}}}.
\end{align*}
Using Lemma \ref{keylem}(1) and (4), we have
\begin{align*}
\lim_{r\rightarrow 1^{-}}\langle A_{1}A_{2}k_{r\xi},k_{r\xi}\rangle
=&\lim_{r\rightarrow 1^{-}}\langle R_{\begin{pmatrix}
\begin{smallmatrix}
 f_1 & 0 \\
0 & \psi_1
\end{smallmatrix}
\end{pmatrix}}
R_{\begin{pmatrix}
\begin{smallmatrix}
 f_2 & 0 \\
0 & \psi_2
\end{smallmatrix}
\end{pmatrix}}k_{r\xi},k_{r\xi}\rangle\\
=&\lim_{r\rightarrow 1^{-}}\langle R_{\begin{pmatrix}
\begin{smallmatrix}
 f_{1}f_2 & 0 \\
0 & \psi_{1}\psi_2
\end{smallmatrix}
\end{pmatrix}}k_{r\xi},k_{r\xi}\rangle\\
=&f_{1}(\xi)\cdot f_{2}(\xi)\\
=&\lim_{r\rightarrow 1^{-}}\langle R_{\begin{pmatrix}
\begin{smallmatrix}
 f_1 & 0 \\
0 & \psi_1
\end{smallmatrix}
\end{pmatrix}}k_{r\xi},k_{r\xi}\rangle\cdot
\lim_{r\rightarrow 1^{-}}\langle R_{\begin{pmatrix}
\begin{smallmatrix}
 f_2 & 0 \\
0 & \psi_2
\end{smallmatrix}
\end{pmatrix}}k_{r\xi},k_{r\xi}\rangle\\
=&\lim_{r\rightarrow 1^{-}}\langle A_{1}k_{r\xi},k_{r\xi}\rangle\cdot
\lim_{r\rightarrow 1^{-}}\langle A_{2}k_{r\xi},k_{r\xi}\rangle \quad a.e. ~on ~\mathbb{T}.
\end{align*}
Similarly,
\begin{align*}
\lim_{r\rightarrow 1^{-}}\langle A_{1}A_{2}\bar{z}\bar{k}_{r\xi},\bar{z}\bar{k}_{r\xi}\rangle
=&\lim_{r\rightarrow 1^{-}}\langle A_{1}\bar{z}\bar{k}_{r\xi},\bar{z}\bar{k}_{r\xi}\rangle\cdot
\lim_{r\rightarrow 1^{-}}\langle A_{2}\bar{z}\bar{k}_{r\xi},\bar{z}\bar{k}_{r\xi}\rangle \quad a.e. ~on ~\mathbb{T}.
\end{align*}
Since the algebraic operations of $L_{2}^{\infty}(\mathbb{T})$ are all performed coordinated-wise,
we have $\rho$ is multiplicative.

By Lemma \ref{keylem}(4), we have $\mathfrak{SR}_{L^{\infty}}\subseteq\ker\rho.$
For every $T=R(f,0,0,\psi)+E\in \ker\rho,$
thus, $f=\psi=0.$
Hence $\mathfrak{SR}_{L^{\infty}}=\ker\rho.$

We define the map
\begin{align*}
\widetilde{\rho} : \mathfrak{R}_{L^{\infty}}/\mathfrak{SR}_{L^{\infty}}&\longrightarrow L_{2}^{\infty}(\mathbb{T}),\\
R_{\begin{pmatrix}
\begin{smallmatrix}
 f & 0 \\
0 & \psi
\end{smallmatrix}
\end{pmatrix}}+\mathfrak{SR}_{L^{\infty}} &\longmapsto (f,\psi).
\end{align*}
Hence, $\widetilde{\rho}$ is a $C^{*}$-isomorphism.
\end{proof}
\begin{corollary}\label{sem}
If $T\in\mathfrak{R}_{L^{\infty}},$ then $\rho(T^{*}T-TT^{*})=(0,0).$
\end{corollary}

\begin{example}
In fact, $\mathfrak{R}_{L^{\infty}}$ is a proper subalgebra of
$B(L^{2}(\mathbb{T})).$ We make some modification to \cite[Example 4]{Englis1995}.
Let $T$ be the operator defined by
\begin{align*}
Tz^{n}=&z^{2n+1},\quad n\in \mathbb{Z}.
\end{align*}
Note that
\begin{align}\label{map}
T^{*}z^{n}=&\rm \left\{ \begin{array}{ll}
   \  z^{\frac{n-1}{2}},&\text{if}~n~\text{is odd}; \\
     \ 0,&\text{if}~n~\text{is even}.
     \end{array} \right.
\end{align}
and
\begin{align*}
(T^{*}T-TT^{*})z^{n}=\rm \left\{ \begin{array}{ll}
   \  0,&\text{if}~n~\text{is odd}; \\
     \ z^{n},&\text{if}~n~\text{is even}.
     \end{array} \right.
\end{align*}
Hence $T^{*}T-TT^{*}$ is the orthogonal projeciton onto span$\{z^{2n}\}_{n\in \mathbb{Z}}.$
\begin{align*}
\langle(T^{*}T-TT^{*})k_{r\xi},k_{r\xi}\rangle
=&(1-r^{2})\langle(T^{*}T-TT^{*})\sum_{i=0}^{\infty}(r\bar{\xi})^{i}z^{i},\sum_{j=0}^{\infty}(r\bar{\xi})^{j}z^{j}\rangle\\
=&(1-r^{2})\langle\sum_{n=0}^{\infty}(r\bar{\xi})^{2n}z^{2n},\sum_{m=0}^{\infty}(r\bar{\xi})^{2m}z^{2m}\rangle\\
=&\langle k_{(r\bar{\xi})^{2}},k_{(r\bar{\xi})^{2}} \rangle\\
=&\frac{1-r^{2}}{1-r^{4}}
=\frac{1}{1+r^{2}}\rightarrow \frac{1}{2}(r\rightarrow 1^{-}).
\end{align*}
By Corollary \ref{sem}, we have $T\notin \mathfrak{R}_{L^{\infty}}.$
\end{example}

\section{$C^{*}-$algebras $\mathfrak{R}_{C(\mathbb{T})}$ and $\mathfrak{R}_{QC}$}

Let $C(\mathbb{T})$ denote the set of continuous complex-valued functions on $\mathbb{T},$ and $C(\mathbb{T})$ is a closed subalgebra of $L^{\infty}.$
The set of all compact operators on $L^{2}(\mathbb{T})$ is denoted by $\mathcal{K}(L^{2}(\mathbb{T})).$
\begin{lemma}\label{SK}
The $C^{*}-$algebra $\mathfrak{R}_{C(\mathbb{T})}$ is irreducible. Furthermore, $LC(L^{2}(\mathbb{T}))\subset \mathfrak{R}_{C(\mathbb{T})}.$
\end{lemma}
\begin{proof}
If $\mathfrak{R}_{C(\mathbb{T})}$ is reducible, then there exists a nontrivial orthogonal projection $Q$ which
commutes with each element of $\mathfrak{R}_{C(\mathbb{T})}.$
In particular, $QR_{\begin{pmatrix}
\begin{smallmatrix}
 z & z \\
z & z
\end{smallmatrix}
\end{pmatrix}}
=R_{\begin{pmatrix}
\begin{smallmatrix}
 z & z \\
z & z
\end{smallmatrix}
\end{pmatrix}}Q$
and $R_{\begin{pmatrix}
\begin{smallmatrix}
 z & z \\
z & z
\end{smallmatrix}
\end{pmatrix}}$ is the bilateral shift. Since the commutant of the
bilateral shift is the set of all multiplications\cite[146]{Halmos1982}, it follows that $Q=M_{\chi_{\Delta}},$ where $\chi_{\Delta}$ is a characteristic function.
Note that
\begin{align*}
R_{\begin{pmatrix}
\begin{smallmatrix}
 z & 0 \\
0 & 0
\end{smallmatrix}
\end{pmatrix}}Q
=& QR_{\begin{pmatrix}
\begin{smallmatrix}
 z & 0 \\
0 & 0
\end{smallmatrix}
\end{pmatrix}},\\
\begin{pmatrix}T_{z}& 0\\
		0 &   0 \end{pmatrix}
\begin{pmatrix}T_{\chi_{\Delta}}& H^{*}_{\chi_{\Delta}}\\
		H_{\chi_{\Delta}}&  \tilde{T}_{\chi_{\Delta}} \end{pmatrix}
=&\begin{pmatrix}T_{z}& 0\\
		0 &   0 \end{pmatrix}
\begin{pmatrix}T_{\chi_{\Delta}}& H^{*}_{\chi_{\Delta}}\\
		H_{\chi_{\Delta}}&   \tilde{T}_{\chi_{\Delta}} \end{pmatrix},\\
\begin{pmatrix}T_{z}T_{\chi_{\Delta}}& T_{z}H^{*}_{\chi_{\Delta}}\\
		0 &   0 \end{pmatrix}
=&\begin{pmatrix}T_{\chi_{\Delta}}T_{z}& 0\\
		T_{\chi_{\Delta}}H_{\chi_{\Delta}} &   0 \end{pmatrix},
\end{align*}
implies
\begin{align*}
T_{z}T_{\chi_{\Delta}}=T_{\chi_{\Delta}}T_{z}.
\end{align*}
Since the commutant of $T_{z}$ is the set of all analytic Toeplitz operators on $H^{2}$ \cite[147]{Halmos1982},
it follows that
$\chi_{\Delta}$ is $0$ or $1$, and $Q=I$ or $Q=0.$ This contradicts our assumption.
Therefore $\mathfrak{R}_{C(\mathbb{T})}$ is irreducible.

Applying the formula $I-T_{z}T_{\bar{z}}=1\otimes 1$ yields
\begin{align*}
R_{\begin{pmatrix}
\begin{smallmatrix}
 1 & 0 \\
0 & 0
\end{smallmatrix}
\end{pmatrix}}-R_{\begin{pmatrix}
\begin{smallmatrix}
 z & 0 \\
0 & 0
\end{smallmatrix}
\end{pmatrix}}R_{\begin{pmatrix}
\begin{smallmatrix}
 \bar{z} & 0 \\
0 & 0
\end{smallmatrix}
\end{pmatrix}}=1\otimes 1.
\end{align*}
where $1\otimes 1$ is an operator of rank 1,
thus $LC(L^{2}(\mathbb{T}))\cap\mathfrak{R}_{C(\mathbb{T})}\neq \{0\}.$
By \cite[5.39]{Douglas1998}, we have $LC(L^{2}(\mathbb{T}))\subset \mathfrak{R}_{C(\mathbb{T})}.$
\end{proof}

The algebra $QC\triangleq(H^{\infty}+C(\mathbb{T}))\cap(\overline{H^{\infty}+C(\mathbb{T})})$
is a closed subalgebra of $L^{\infty}(\mathbb{T})$ which properly contains $C(\mathbb{T}).$
Let $\mathfrak{SR}_{QC}$(resp. $\mathfrak{SR}_{C(\mathbb{T})}$) be the closed ideal of
$\mathfrak{R}_{QC}$(resp. $\mathfrak{R}_{C(\mathbb{T})}$) generated by operators of
the form
\begin{align}\label{semi}
R_{\begin{pmatrix}
\begin{smallmatrix}
 f_{1} & \phi_{1} \\
g_{1} & \psi_{1}
\end{smallmatrix}
\end{pmatrix}}
R_{\begin{pmatrix}
\begin{smallmatrix}
 f_{2} & \phi_{2} \\
g_{2} & \psi_{2}
\end{smallmatrix}
\end{pmatrix}}-
R_{\begin{pmatrix}\begin{smallmatrix}
 f_{1}f_{2} & \phi \\
g & \psi_{1}\psi_{2}
\end{smallmatrix}
\end{pmatrix}}
\end{align}
where$f_{i},g_{i},\phi_{i},\psi_{i},
g,\phi$ are in $QC$(resp. $C(\mathbb{T})$)$(i=1,2).$
\begin{lemma}\label{K}
$\mathfrak{SR}_{C(\mathbb{T})}=\mathcal{K}(L^{2}(\mathbb{T})),$ and
$\mathfrak{SR}_{QC}=\mathcal{K}(L^{2}(\mathbb{T})).$
\end{lemma}
\begin{proof}
If $f_{i},g_{i},\phi_{i},\psi_{i},g,\phi\in C(\mathbb{T})$(resp. $QC$)(i=1,2),
an easy computation shows that
\begin{align*}
&R_{\begin{pmatrix}
\begin{smallmatrix}
 f_{1} & \phi_{1} \\
g_{1} & \psi_{1}
\end{smallmatrix}
\end{pmatrix}}
R_{\begin{pmatrix}
\begin{smallmatrix}
 f_{2} & \phi_{2} \\
g_{2} & \psi_{2}
\end{smallmatrix}
\end{pmatrix}}-
R_{\begin{pmatrix}\begin{smallmatrix}
 f_{1}f_{2} & \phi \\
g & \psi_{1}\psi_{2}
\end{smallmatrix}
\end{pmatrix}}\\
=&\begin{pmatrix}T_{f_{1}}T_{f_{2}}+H^{*}_{\bar{\phi_{1}}}H_{g_{2}}-T_{f_{1}f_{2}}& T_{f_{1}}H^{*}_{\bar{\phi_{2}}}+H^{*}_{\bar{\phi_{1}}}\tilde{T}_{\psi_{2}}-H^{*}_{\bar{\phi}}\\
   H_{g_{1}}T_{f_{2}}+\tilde{T}_{\psi_{1}}H_{g_{2}}-H_{g} &  H_{g_{1}}H^{*}_{\bar{\phi_{2}}}+\tilde{T}_{\psi_{1}}\tilde{T}_{\psi_{2}}-\tilde{T}_{\psi_{1}\psi_{2}}
\end{pmatrix}\quad\\
=&\begin{pmatrix}H^{*}_{\bar{\phi_{1}}}H_{g_{2}}-H^{*}_{\bar{f_{1}}}H_{f_{2}}& T_{f_{1}}H^{*}_{\bar{\phi_{2}}}+H^{*}_{\bar{\phi_{1}}}\tilde{T}_{\psi_{2}}-H^{*}_{\bar{\phi}}\\
   H_{g_{1}}T_{f_{2}}+\tilde{T}_{\psi_{1}}H_{g_{2}}-H_{g} &  H_{g_{1}}H^{*}_{\bar{\phi_{2}}}-H_{\psi_{1}}H^{*}_{\bar{\psi_{2}}}
\end{pmatrix}.\quad
\end{align*}

The second equality follows form the formulas $T_{ab}-T_{a}T_{b}=H^{*}_{\bar{a}}H_{b}$
and $\tilde{T}_{ab}-\tilde{T}_{a}\tilde{T}_{b}=H_{a}H^{*}_{\bar{b}}.$
Since the Hankel operator $H_{\varphi}$ is compact if and only if $\varphi\in H^{\infty}+C(\mathbb{T})$
by\cite[p.27]{Peller2003}, it follows that
\begin{align*}
R_{\begin{pmatrix}
\begin{smallmatrix}
 f_{1} & \phi_{1} \\
g_{1} & \psi_{1}
\end{smallmatrix}
\end{pmatrix}}
R_{\begin{pmatrix}
\begin{smallmatrix}
 f_{2} & \phi_{2} \\
g_{2} & \psi_{2}
\end{smallmatrix}
\end{pmatrix}}-
R_{\begin{pmatrix}\begin{smallmatrix}
 f_{1}f_{2} & \phi \\
g & \psi_{1}\psi_{2}
\end{smallmatrix}
\end{pmatrix}}
\end{align*}
is compact, and
$\mathfrak{SR}_{C(\mathbb{T})}\subset \mathcal{K}(L^{2}(\mathbb{T}))(\text{resp.} \mathfrak{SR}_{QC)}\subset \mathcal{K}(L^{2}(\mathbb{T})).$ On the other hand,
$LC(L^{2}(\mathbb{T}))$ contains no proper closed ideal. Hence, $\mathfrak{SR}_{C(\mathbb{T})}=LC(L^{2}(\mathbb{T}))$ (resp.$ \mathfrak{SR}_{QC)}=\mathcal{K}(L^{2}(\mathbb{T}))$).
\end{proof}
\begin{corollary}
For every $T\in \mathfrak{R}_{L^{\infty}},$ we have
\begin{align*}
\|\rho(T)\|\leq\|T\|_{e}.
\end{align*}
In particular, if $H=\scriptsize{\begin{pmatrix}f& \phi\\
		g &  f \end{pmatrix}},$ then
\begin{align*}
max\{\|f\|_{\infty},\|\psi\|_{\infty}\}\leq\|R_H\|_{e}.
\end{align*}
\end{corollary}
\begin{proof}
If $T\in \mathfrak{R}_{L^{\infty}},$ by Theorem \ref{main}, we have
\begin{align*}
\inf_{A\in \mathfrak{SR}_{L^{\infty}}}\|T+A\|=\|\rho(T)\|.
\end{align*}
On the other hand, $\mathcal{K}\subset \mathfrak{SR}_{L^{\infty}}$ by Lemma \ref{K}.
Therefore,
\begin{align*}
\inf_{A\in \mathfrak{SR}_{L^{\infty}}}\|T+A\|\leq\inf_{K\in \mathcal{K}(L^{2}(\mathbb{T}))}\|T+K\|=\|T\|_{e}.
\end{align*}
Use Theorem \ref{main} again,
\begin{align*}
\|\rho(R_H)\|=max\{\|f\|_{\infty},\|\psi\|_{\infty}\}.
\end{align*}
\end{proof}
If $T$ is a bounded linear operator on Hilbert space $H,$
$\sigma_{e}(T)$ denotes the essential spectrum of $T.$
For $\varphi\in L^{\infty},\emph{Ran}_{ess}{\varphi}$ denotes
the essential range of $\varphi.$
If $E$ is a subset of complex plane $\mathbb{C},$ the convex hull of $E$ will be denoted by
$coE.$
Combining Theorem \ref{main} and Lemma \ref{K}, we get the following result.
\begin{corollary}\label{fredholm}
There exists a *-homomorphism $\zeta$ from the quotient algebra $\mathfrak{R}_{L^{\infty}}/\mathcal{K}$  onto $L_{2}^{\infty}(\mathbb{T})$ such that  the diagram

\[\xymatrix{
\mathfrak{R}_{L^{\infty}} \ar[rr]^{\pi}\ar[dr]_{\rho} &   & \mathfrak{R}_{L^{\infty}}/\mathcal{K}(L^{2}(\mathbb{T}))\ar[dl]^{\zeta} \\
                         & L_{2}^{\infty}(\mathbb{T}) &
}\]
commutes. Moreover,
\begin{enumerate}
\item For every $T\in \mathfrak{R}_{L^{\infty}},$  if $T$ is Fredholm, then $\rho(T)$ is invertible in
$L_{2}^{\infty}(\mathbb{T});$
\item $\emph{Ran}_{ess}{f} \cup \emph{Ran}_{ess}{\psi}\subset \sigma_{e}(R_{H}).$
\end{enumerate}
\end{corollary}
Recall the spectral inclusion theorem of Toeplitz operator\cite{Douglas1998},
\begin{align}\label{inclusion}
\emph{Ran}_{ess}{f} \subset \sigma_{e}(T_{f})\subset \sigma(T_{f})\subset co\emph{Ran}_{ess}{f}.
\end{align}
Corollary \ref{fredholm} give the first inclusion similar to \eqref{inclusion}, the next theorem will show the third inclusion similar to \eqref{inclusion}.

\begin{proposition}
Let $H=\scriptsize{\begin{pmatrix}f_{1}& \phi\\
		g &  f_{2} \end{pmatrix}}\in L^{2\times2}_{\infty}(\mathbb{T}).$ If
we define
\[\mathcal{G}_{i}=co\emph{Ran}_{ess}{f_{i}}\cup\{\lambda\notin\emph{Ran}_{ess}{f_{i}}:
d_{i}(\lambda)\leq\delta\|(f_{i}-\lambda)^{-1}\|_{\infty}\}\]
where
\begin{align*}
d_{i}(\lambda)=(1-dist((f_{i}-\lambda)/|f_{i}-\lambda|,H^{\infty})^{2})^{1/2},
\quad \delta=\min \{dist(\bar{\phi},H^{\infty}),dist(g,H^{\infty})\}
\end{align*}
for $n=1,2,$ then
\begin{align*}
\sigma(R_{H})\subset \mathcal{G}_{1}\cup\mathcal{G}_{2}.
\end{align*}
\end{proposition}
\begin{proof}
Suppose $\lambda \in \rho(T_{f_{1}})\cap \rho(\tilde{T}_{f_{2}}),$ we have
\begin{align*}
R_{H}-\lambda I_{L^2}
=&\begin{pmatrix}T_{f_{1}-\lambda}& H^{*}_{\bar {\phi}}\\
		H_{g} &   \tilde{T}_{f_{2}-\lambda} \end{pmatrix}\\
=&\begin{pmatrix}T_{f_{1}-\lambda}& 0\\
		0 &   \tilde{T}_{f_{2}-\lambda} \end{pmatrix}
+\begin{pmatrix}0& H^{*}_{\bar {\phi}}\\
		H_{g} &  0 \end{pmatrix}\\
=&\begin{pmatrix}T_{f_{1}-\lambda}& 0\\
		0 &   \tilde{T}_{f_{2}-\lambda} \end{pmatrix}
\bigg(I_{L^2}+\begin{pmatrix}0& T^{-1}_{f_{1}-\lambda}H^{*}_{\bar {\phi}}\\
		\tilde{T}^{-1}_{f_{2}-\lambda}H_{g} &   0 \end{pmatrix}\bigg).
\end{align*}
If $\|H_{\bar{\phi}}\|<\|T^{-1}_{f_{1}-\lambda}\|^{-1}$
and $\|H_{g}\|<\|\tilde{T}^{-1}_{f_{2}-\lambda}\|^{-1},$ then
\begin{align*}
\bigg\|\begin{pmatrix}0& T^{-1}_{f_{1}-\lambda}H^{*}_{\bar {\phi}}\\
		\tilde{T}^{-1}_{f_{2}-\lambda}H_{g} &   0 \end{pmatrix}\bigg\|
=\max\{\|T^{-1}_{f_{1}-\lambda}H^{*}_{\bar {\phi}}\|,\|\tilde{T}^{-1}_{f_{2}-\lambda}H_{g}\|\}<1,
\end{align*}
and $\lambda \in \rho (R_{H}).$
This mean that
\begin{align*}
\{\lambda\in\rho(T_{f_{1}}):\|H_{\bar{\phi}}\|<\|T^{-1}_{f_{1}-\lambda}\|^{-1}\}
\cap
\{\lambda\in\rho(\tilde{T}_{f_{2}}):\|H_{g}\|<\|\tilde{T}^{-1}_{f_{2}-\lambda}\|^{-1}\}
\subset \rho (R_{H})
\end{align*}
or
\begin{equation}\label{R}
\begin{split}
\sigma (R_{H})\subset &\sigma(T_{f_{1}})\cup \{\lambda\in\rho(T_{f_{1}}):\|T^{-1}_{f_{1}-\lambda}\|^{-1}\leq\|H_{\bar{\phi}}\|\}\\
&\cup \sigma(\tilde{T}_{f_{2}})\cup
\{\lambda\in\rho(\tilde{T}_{f_{2}}):\|\tilde{T}^{-1}_{f_{2}-\lambda}\|^{-1}\leq\|H_{g}\|\}.
\end{split}
\end{equation}
Repeat the above reasoning for $R^{*}_{H},$ we have
\begin{align*}
\sigma (R^{*}_{H})\subset &\sigma(T_{\bar{f_{1}}})\cup \{\lambda\in\rho(T_{\bar{f_{1}}}):\|T^{-1}_{\bar{f_{1}}-\lambda}\|^{-1}\leq\|H_{g}\|\}\\
&\cup \sigma(\tilde{T}_{\bar{f_{2}}})\cup
\{\lambda\in\rho(\tilde{T}_{\bar{f_{2}}}):\|\tilde{T}^{-1}_{\bar{f_{2}}-\lambda}\|^{-1}\leq\|H_{\bar{\phi}}\|\}.
\end{align*}
Taking conjugates, we get
\begin{align*}
\sigma (R_{H})\subset &\sigma(T_{f_{1}})\cup \{\bar{\lambda}\in\rho(T_{\bar{f_{1}}}):\|T^{-1}_{\bar{f_{1}}-\bar{\lambda}}\|^{-1}\leq\|H_{g}\|\}\\
&\cup \sigma(\tilde{T}_{\bar{f_{2}}})\cup
\{\bar{\lambda}\in\rho(\tilde{T}_{\bar{f_{2}}}):\|\tilde{T}^{-1}_{\bar{f_{2}}-\bar{\lambda}}\|^{-1}\leq\|H_{\bar{\phi}}\|\}.
\end{align*}
Since $\|T^{-1}_{\bar{f_{1}}-\bar{\lambda}}\|
=\|(T^{-1}_{f_{1}-\lambda})^{*}\|=\|(T^{-1}_{f_{1}-\lambda})\|$ and
$\|\tilde{T}^{-1}_{\bar{f_{2}}-\bar{\lambda}}\|
=\|(\tilde{T}^{-1}_{f_{2}-\lambda})^{*}\|=\|(\tilde{T}^{-1}_{f_{2}-\lambda})\|,$ it follows that
\begin{equation}\label{Rstar}
\begin{split}
\sigma (R_{H})\subset &\sigma(T_{f_{1}})\cup \{\lambda\in\rho(T_{f_{1}}):\|T^{-1}_{f_{1}-\lambda}\|^{-1}\leq\|H_{g}\|\}\\
&\cup \sigma(\tilde{T}_{f_{2}})\cup
\{\lambda\in\rho(\tilde{T}_{f_{2}}):\|\tilde{T}^{-1}_{f_{2}-\lambda}\|^{-1}\leq\|H_{\bar{\phi}}\|\}.
\end{split}
\end{equation}
According the norm of Hankel operator (\cite[Theorem 1.4]{Peller2003}), we have
$\|H_{\bar{\phi}}\|=dist(\bar{\phi},H^{\infty})$ and
$\|H_{g}\|=dist(g,H^{\infty}).$
Let $\delta=\min \{dist(\bar{\phi},H^{\infty}),dist(g,H^{\infty})\}.$ We combine \eqref{R} and \eqref{Rstar}. Thus
\begin{align*}
\sigma (R_{H})\subset &\sigma(T_{f_{1}})\cup \{\lambda\in\rho(T_{f_{1}}):\|T^{-1}_{f_{1}-\lambda}\|^{-1}\leq\delta\}\\
&\cap \sigma(\tilde{T}_{f_{2}})\cap
\{\lambda\in\rho(\tilde{T}_{f_{2}}):\|\tilde{T}^{-1}_{f_{2}-\lambda}\|^{-1}\leq\delta\|\}.
\end{align*}
Since $\tilde{T}_{f_{2}}$ and $T^{*}_{f_{2}}$ are anti-unitary,
$\sigma(\tilde{T}_{f_{2}})=\sigma(T_{f_{2}})$ and $\|\tilde{T}^{-1}_{f_{2}-\lambda}\|=\|T^{-1}_{f_{2}-\lambda}\|.$
Using the \eqref{inclusion} and norm estimation of the inverse of Toeplitz operator\cite[page.125.]{nikolski2020}
\begin{align*}
\frac{(1-dist(\varphi/|\varphi|,H^{\infty})^{2})^{1/2}}{\|\varphi^{-1}\|}\leq\|T^{-1}_{\varphi}\|^{-1},
\end{align*}
we have
\begin{align*}
\{\lambda\in\rho(T_{f_{1}}):\|T^{-1}_{f_{1}-\lambda}\|^{-1}\leq\delta\}
\subset & \{\lambda\notin\emph{Ran}_{ess}{f_{1}}:
d_{1}(\lambda)\leq\delta\|(f_{1}-\lambda)^{-1}\|_{\infty}\},\\
\{\lambda\in\rho(T_{f_{2}}):\|T^{-1}_{f_{2}-\lambda}\|^{-1}\leq\delta\}
\subset &\{\lambda\notin\emph{Ran}_{ess}{f_{2}}:
d_{2}(\lambda)\leq\delta\|(f_{2}-\lambda)^{-1}\|_{\infty}\},\\
\sigma(T_{f_{1}})\subset &co\emph{Ran}_{ess}{f_{1}},\\ \quad and \quad \sigma(T_{f_{2}})\subset & co\emph{Ran}_{ess}{f_{2}},
\end{align*}
where $d_{i}(\lambda)=(1-dist((f_{i}-\lambda)/|f_{i}-\lambda|,H^{\infty})^{2})^{1/2},i=1,2.$
\end{proof}

\begin{theorem}\label{CC}
The sequence
\begin{align*}
0 \longrightarrow \mathcal{K}(L^{2}(\mathbb{T})) \longrightarrow \mathfrak{R}_{C(\mathbb{T})}
\longrightarrow C_{2}(\mathbb{T})\longrightarrow 0
\end{align*}
is a short exact sequence; that is, the quotient algebra
$\mathfrak{R}_{C(\mathbb{T})}/\mathcal{K}$
is *-isometrically isomorphic to $C_{2}(\mathbb{T}).$
\end{theorem}
\begin{proof}
Using the proof of Theorem \ref{main} and Lemma \ref{K} , for every operator $T\in\mathfrak{R}_{C(\mathbb{T})}$ have the following form
\begin{align}\label{CB}
T=R_{\begin{pmatrix}
\begin{smallmatrix}
 f & 0 \\
0 & \psi
\end{smallmatrix}
\end{pmatrix}}+K,\quad f,\psi\in C(\mathbb{T}), K\in\mathcal{K}.
\end{align}
The map $\tilde{\rho}$ defined in \eqref{map} is  *-isometrically isomorphic
from $\mathfrak{R}_{C(\mathbb{T})}/\mathcal{K}(L^{2}(\mathbb{T}))$ to $C_{2}(\mathbb{T}).$
\end{proof}
\begin{remark}
In fact, the previous theorem can be extend to the algebra $QC.$
The sequence
\begin{align*}
0 \longrightarrow \mathcal{K}(L^{2}(\mathbb{T})) \longrightarrow \mathfrak{R}_{QC}
\longrightarrow QC_{2}\longrightarrow 0
\end{align*}
is a short exact sequence. The proof is similar in spirit to Theorem \ref{CC}.
\end{remark}
\begin{corollary}
For every $T\in \mathfrak{R}_{QC},$ we have
\begin{align*}
\|\rho(T)\|=\|T\|_{e}.
\end{align*}
In particular, if $H=\scriptsize{\begin{pmatrix}f& \phi\\
		g &  f \end{pmatrix}}\in QC_{2\times 2},$  then
\begin{align*}
max\{\|f\|_{\infty},\|\psi\|_{\infty}\}=\|R_H\|_{e}.
\end{align*}
\end{corollary}
\begin{corollary}\label{QC}
If $H=\scriptsize{\begin{pmatrix}f& \phi\\
		g &  f \end{pmatrix}}\in QC_{2\times 2},$ then $\sigma_{e}(R_{H})=\emph{Ran}_{ess}{f} \cup \emph{Ran}_{ess}{\psi}.$
Moreover, $R_H$ is Fredholm if and only if $f$ and $\psi$ are
invertible in $QC.$
\end{corollary}
\begin{remark}\label{cf}
If $H=\scriptsize{\begin{pmatrix}f& \phi\\
		g &  f \end{pmatrix}}\in C(\mathbb{T})_{2\times 2}$, then $\sigma_{e}(R_{H})=f(\mathbb{T}) \cup \psi(\mathbb{T}) .$
\end{remark}
\begin{definition}
Let $f$ is an invertible function in $C(\mathbb{T}),$ the winding number of $f$ about the origin is defined by
\[\sharp(f)=\frac{1}{2\pi i}\int_{f(\mathbb{T})}\frac{dz}{z}.\]
\end{definition}
\begin{definition}
Let $T$ be a bounded linear operator on Hilbert space $H,$ a bounded linear operator
$B$ on $H$ is called the regularizer of $T$ if $BT-I$ and $TB-I$ are compact.
If $T$ is Fredholm, the difference
$\operatorname{ind} T=\operatorname{dim} \operatorname{ker} T-\operatorname{dim} \operatorname{ker} T^{*}$ is call the index of $T.$
\end{definition}
\begin{corollary}\label{i}
If $T$ is Fredholm operator in $\mathfrak{R}_{C(\mathbb{T})},$ then
\begin{enumerate}
  \item $R_{\tiny\begin{pmatrix}
\begin{smallmatrix}
 f_{0}^{-1} & 0 \\
0 & \psi_{0}^{-1}
\end{smallmatrix}
\end{pmatrix}}$ is a regularizer of $T;$
\item $ind (T)=\sharp(\psi_{0})-\sharp(f_{0}),$
\end{enumerate}
where $f_{0}(\xi)=\lim_{r\rightarrow 1^{-}}\langle Tk_{r\xi},k_{r\xi}\rangle,
\psi_{0}(\xi)=\lim_{r\rightarrow 1^{-}}\langle T\bar{z}\bar{k}_{r\xi},\bar{z}\bar{k}_{r\xi}\rangle.$

In particular, if $H=\scriptsize{\begin{pmatrix}f& \phi\\
		g &  \psi \end{pmatrix}}\in C(\mathbb{T})_{2\times 2}$ and $R_{H}$ is a Fredholm operator,
then
\begin{align*}
\text{ind} (R_{H})=\sharp(\psi)-\sharp(f).
\end{align*}
\end{corollary}
\begin{proof}
If $T\in \mathfrak{R}_{C(\mathbb{T})},$
by the formula \eqref{CB}, we have
\begin{align*}
T=R_{\begin{pmatrix}
\begin{smallmatrix}
 f_0 & 0 \\
0 & \psi_0
\end{smallmatrix}
\end{pmatrix}}+K,\quad f_{0},\psi_{0}\in C(\mathbb{T}), \quad K\in\mathcal{K}(L^{2}(\mathbb{T})).
\end{align*}
where $f_{0}(\xi)=\lim_{r\rightarrow 1^{-}}\langle Tk_{r\xi},k_{r\xi}\rangle,
\psi_{0}(\xi)=\lim_{r\rightarrow 1^{-}}\langle T\bar{z}\bar{k}_{r\xi},\bar{z}\bar{k}_{r\xi}\rangle,$
and hence $f$ and $\psi$ are invertible in $C(\mathbb{T})$ by the remark \ref{cf}.
A  calculation shows that
\begin{align*}
R_{\tiny\begin{pmatrix}
\begin{smallmatrix}
 f_{0}^{-1} & 0 \\
0 & \psi_{0}^{-1}
\end{smallmatrix}
\end{pmatrix}}T=&R_{\tiny\begin{pmatrix}
\begin{smallmatrix}
 f_{0}^{-1} & 0 \\
0 & \psi_{0}^{-1}
\end{smallmatrix}
\end{pmatrix}}R_{\begin{pmatrix}
\begin{smallmatrix}
 f_0 & 0 \\
0 & \psi_0
\end{smallmatrix}
\end{pmatrix}}
+R_{\begin{pmatrix}
\begin{smallmatrix}
 f_0 & 0 \\
0 & \psi_0
\end{smallmatrix}
\end{pmatrix}}K\\
=&\begin{pmatrix}T_{f_{0}^{-1}}& 0\\
		0 &   \tilde{T}_{\psi_{0}^{-1}} \end{pmatrix}
\begin{pmatrix}T_{f_{0}}& 0\\
		0 &   \tilde{T}_{\psi_{0}} \end{pmatrix}
+R_{\begin{pmatrix}
\begin{smallmatrix}
 f_0 & 0 \\
0 & \psi_0
\end{smallmatrix}
\end{pmatrix}}K\\
=&\begin{pmatrix}T_{f_{0}^{-1}}T_{f_{0}}& 0\\
		0 &   \tilde{T}_{\psi_{0}^{-1}} \tilde{T}_{\psi_{0}}  \end{pmatrix}+R_{\begin{pmatrix}
\begin{smallmatrix}
 f_0 & 0 \\
0 & \psi_0
\end{smallmatrix}
\end{pmatrix}}K\\
=&\begin{pmatrix}I_{H^{2}}-H^{*}_{\overline{f_{0}^{-1}}}H_{f_{0}}& 0\\
		0 &   I_{\bar{z}\overline{H^{2}}}-H_{\psi_{0}^{-1}}H^{*}_{\bar{\psi_{0}}}\end{pmatrix}
+R_{\begin{pmatrix}
\begin{smallmatrix}
 f_0 & 0 \\
0 & \psi_0
\end{smallmatrix}
\end{pmatrix}}K\\
=&I+\begin{pmatrix}-H^{*}_{\overline{f_{0}^{-1}}}H_{f_{0}}& 0\\
		0 &   -H_{\psi_{0}^{-1}}H^{*}_{\bar{\psi_{0}}}\end{pmatrix}
+R_{\begin{pmatrix}
\begin{smallmatrix}
 f_0 & 0 \\
0 & \psi_0
\end{smallmatrix}
\end{pmatrix}}K.\\
\end{align*}
Since the Hankel operator $H_{\varphi}$ is compact if and only if $\varphi\in H^{\infty}+C(\mathbb{T})$
by \cite[p.27]{Peller2003}, we have $H^{*}_{\overline{f_{0}^{-1}}}H_{f_{0}}$ and $H_{\psi_{0}^{-1}}H^{*}_{\bar{\psi_{0}}}$ are
compact, so $R_{\tiny\begin{pmatrix}
\begin{smallmatrix}
 f_{0}^{-1} & 0 \\
0 & \psi_{0}^{-1}
\end{smallmatrix}
\end{pmatrix}}T-I$ is compact, similarly,
$TR_{\tiny\begin{pmatrix}
\begin{smallmatrix}
 f_{0}^{-1} & 0 \\
0 & \psi_{0}^{-1}
\end{smallmatrix}
\end{pmatrix}}-I$ is compact.

Since the Fredholm index is stable under compact operator
perturbations\cite[p.98]{Arveson2002}, it follows that
\begin{align*}
\text{ind} (T)&= \text{ind} (R_{\begin{pmatrix}
\begin{smallmatrix}
 f_0 & 0 \\
0 & \psi_0
\end{smallmatrix}
\end{pmatrix}}+K)\\
&= \text{ind} R_{\begin{pmatrix}
\begin{smallmatrix}
 f_0 & 0 \\
0 & \psi_0
\end{smallmatrix}
\end{pmatrix}}\\
&=\text{ind} \begin{pmatrix}T_{f_{0}}& 0\\
		0 &   \tilde{T}_{\psi_{0}} \end{pmatrix}\\
&=\text{ind}~T_{f_{0}}+ \text{ind}~\tilde{T}_{\psi_{0}}.
\end{align*}
Note that
\begin{align*}
\text{ind}~\tilde{T}_{\psi_{0}}&=\dim \ker(\tilde{T}_{\psi_{0}})-\dim \ker (\tilde{T}^{*}_{\psi_{0}})\\
&=\dim \ker(VT^{*}_{\psi_{0}}V)-\dim \ker (VT_{\psi_{0}}V)\\
&=\dim \ker(T^{*}_{\psi_{0}})-\dim \ker (T_{\psi_{0}})\\
&=-\text{ind}~T_{\psi_{0}}.
\end{align*}
By the theorem \cite[7,26]{Douglas1998}, we have
$\text{ind}~T_{f_{0}}=-\sharp(f_{0})$ and $\text{ind}~\tilde{T}_{\psi_{0}}=\sharp(\psi_{0}).$
Therefore, $\text{ind}(T)=\sharp(\psi_{0})-\sharp(f_{0}).$
\end{proof}
\begin{corollary}
If $H=\scriptsize{\begin{pmatrix}f& \phi\\
		g &  \psi \end{pmatrix}}\in C(\mathbb{T})_{2\times 2},$
then $R_{H}$ is invertible
if and only if the following conditions hold:
\begin{enumerate}
\item $f$ and $\phi$ are invertible,
\item $\sharp(\psi)=\sharp(f),$ and
\item either $\operatorname{Ker}(R_{H})=\{0\}$ or $\operatorname{Ker}(R^{*}_{H})=\{0\}.$
\end{enumerate}
\end{corollary}
\begin{proof}
By Corollary \ref{fredholm}, we have $\emph{Ran}_{ess}{f} \cup \emph{Ran}_{ess}{\psi}\subset \sigma(R_{H}).$
Suppose that $R_{H}$ is invertible, then
\begin{enumerate}[(a)]
\item $f$ and $\phi$ are invertible;
\item $\operatorname{Ker}(R_{H})=\{0\}$ and $\operatorname{Ker}(R^{*}_{H})=\{0\}.$
\end{enumerate}
It follows that $\text{ind}~ (R_{H})=0.$
By Corollary \ref{i}, we have $\sharp(\psi)=\sharp(f);$

On the other hand,
if $R_{H}$ is Fredholm, then $R_{H}$ is invertible if and only if
\begin{enumerate}[(i)]
\item ind$(R_{H})=0$;
\item either $\operatorname{Ker}(R_H)=\{0\}$ or $\operatorname{Ker}(R^{*}_{H})=\{0\}.$
\end{enumerate}
By Remark \ref{cf}, $R_{H}$ is Fredholm if and only if $f$ and $\psi$ are
invertible, hence the result follows.
\end{proof}
\begin{remark}
There exist some examples showing that both of $\operatorname{Ker}(R_{H})$ and $\operatorname{Ker}(R^{*}_H)$
are nontrivial. For example, if $u$ and $\theta$ are nonconstant inner functions,
then
\[\bar{z}\overline{(H^{2}\ominus \theta H^{2})}\subseteq \operatorname{Ker} R_{\begin{pmatrix}
\begin{smallmatrix}
 u & 0 \\
0 & \theta
\end{smallmatrix}
\end{pmatrix}}\] and
\[H^{2}\ominus uH^{2}\subseteq \operatorname{Ker} R^{*}_{\begin{pmatrix}
\begin{smallmatrix}
 u & 0 \\
0 & \theta
\end{smallmatrix}
\end{pmatrix}}.\]
Let $\Delta$ is a proper subset of $\mathbb{T}$ and has positive measure, $\chi_{\Delta}$ is the characteristic function of $\Delta,$ we have
$R_{\begin{pmatrix}
\begin{smallmatrix}
 \chi_{\Delta} & \chi_{\Delta} \\
\chi_{\Delta} & \chi_{\Delta}
\end{smallmatrix}
\end{pmatrix}}=M_{\chi_{\Delta}}$ and
$\dim \ker M_{\chi_{\Delta}}=\dim \ker M^{*}_{\chi_{\Delta}}=\infty.$
\end{remark}

\section{invertible and Fredholm of GISO}
In this section, we found that GSIOs and
singular integral operators with $2\times2$ matrix symbol are equivalent after extension.

\begin{definition}\cite{bart1992matricial}
Let $T$ and $S$ are bounded operator on Hilbert space $\mathcal{H}_1$
and $\mathcal{H}_2$ respectively.
The operators $T$ and $S$ are called equivalent after extension, written
$T\stackrel{\ast}{\sim}S,$ if there exist Hilbert spaces $Z$ and $W$ such that $T \oplus I_{Z}$ and $S \oplus I_{W}$ are equivalent operators. This means that there exist invertible bounded linear operators $E$ and $F$ such that
\begin{align*}
\left(\begin{array}{ll}
T & 0 \\
0 & I_{Z}
\end{array}\right)=E\left(\begin{array}{ll}
S & 0 \\
0 & I_{W}
\end{array}\right) F.
\end{align*}
The relation $\stackrel{*}{\sim}$ is reflexive, symmetric and transitive.
\end{definition}

\begin{theorem}\cite{bart1992matricial}\label{iff}
If $T \stackrel{\star}{\sim}S,$ then
$T$ is invertible(Fredholm) if and only if $S$ is invertible (Fredholm).
\end{theorem}

Let \begin{align*}
A=\begin{pmatrix}f& 0\\
		g &   -1 \end{pmatrix},
B=\begin{pmatrix}\varphi& -1\\
		\psi &   0 \end{pmatrix},
\end{align*}
where $f,g,\phi,\psi\in L^{\infty}(\mathbb{T}).$
Write the Cauchy singular integral operators with $2\times2$ matrix symbol
\begin{equation}\label{AB}
\begin{split}
A\mathbb{P}_{+}+B\mathbb{P}_{-}
=&\begin{pmatrix}f& 0\\
		g &   -1 \end{pmatrix}
\begin{pmatrix}P_{+}& 0\\
		0 &   P_{+} \end{pmatrix}+
\begin{pmatrix}\varphi& -1\\
		\psi &   0 \end{pmatrix}
\begin{pmatrix}P_{-}& 0\\
		0 &   P_{-} \end{pmatrix}:L_{2}^{2}(\mathbb{T})\rightarrow L_{2}^{2}(\mathbb{T}).
\end{split}
\end{equation}
\begin{theorem}\label{q}
Let $H={\tiny\begin{pmatrix}
   f& \phi\\
   g & \psi \end{pmatrix}}\in L_{2\times2}^{\infty}(\mathbb{T}),$
$R_H\stackrel{\ast}{\sim} A\mathbb{P}_{+}+B\mathbb{P}_{-}.$
\end{theorem}
\begin{proof}
Let $H_{1}={\begin{pmatrix}
\begin{smallmatrix}
 g & \psi \\
f & \phi
\end{smallmatrix}
\end{pmatrix}},$ an easy computation shows that
\begin{align*}
&\begin{pmatrix}P_{+}& P_{-}\\
		P_{-} &   P_{+} \end{pmatrix}
(A\mathbb{P}_{+}+B\mathbb{P}_{-})
\begin{pmatrix}I_{L^{2}}& 0\\
		R_{H_1} &   -I_{L^{2}} \end{pmatrix}\\
=&\begin{pmatrix}P_{+}& P_{-}\\
		P_{-} &   P_{+} \end{pmatrix}
\begin{pmatrix} fP_{+}+\varphi P_{-} & -P_{-}\\
		gP_{+}+{\psi}P_{-} &  -P_{+}  \end{pmatrix}
\begin{pmatrix}I_{L^{2}}& 0\\
		R_{H_{1}} &   -I_{L^{2}} \end{pmatrix}\\
=&\begin{pmatrix}P_{+}fP_{+}+P_{+}\phi P_{-}+P_{-}gP_{+}
+P_{-}\psi P_{-}& 0\\
P_{-}fP_{+}+P_{-}\psi P_{-}
+P_{+}g P_{+}+P_{+}\psi P_{-} &   -I_{L^{2}} \end{pmatrix}
\begin{pmatrix}I_{L^{2}}& 0\\
		R_{H_{1}} &   -I_{L^{2}} \end{pmatrix}\\
=&\begin{pmatrix}R_H& 0\\
		R_{H_{1}} &   -I_{L^{2}} \end{pmatrix}
\begin{pmatrix}I_{L^{2}}& 0\\
		R_{H_{1}} &   -I_{L^{2}} \end{pmatrix}\\
=&\begin{pmatrix}R_H& 0\\
		0 &   I_{L^{2}} \end{pmatrix}.
\end{align*}
The operators
$\begin{pmatrix}P_{+}& P_{-}\\
		P_{-} &   P_{+} \end{pmatrix}$ and
$\begin{pmatrix}I_{L^{2}}& 0\\
		R_{H_{1}} &   -I_{L^{2}} \end{pmatrix}$
are invertible, and
\begin{align*}
&\begin{pmatrix}P_{+}& P_{-}\\
		P_{-} &   P_{+} \end{pmatrix}^{-1}
=\begin{pmatrix}P_{+}& P_{-}\\
		P_{-} &   P_{+} \end{pmatrix},\\
&\begin{pmatrix}I_{L^{2}}& 0\\
		R_{H_1} &   -I_{L^{2}} \end{pmatrix}^{-1}
=\begin{pmatrix}I_{L^{2}}& 0\\
		R_{H_1} &   -I_{L^{2}} \end{pmatrix}.
\end{align*}
Hence the operators $R_H$ and
$A\mathbb{P}_{+}+B\mathbb{P}_{-}$
are equivalent after extension.
\end{proof}

If $f$ and $\psi$ are invertible, then
$A$ and $B$ are invertible and
\begin{align*}
A^{-1}=\begin{pmatrix}f^{-1}& 0\\
		f^{-1}g &   -1 \end{pmatrix},\quad
B^{-1}=\begin{pmatrix}0& \psi^{-1}\\
		-1 &   \phi\psi^{-1} \end{pmatrix},
\end{align*}
In this case
\begin{align*}
&A\mathbb{P}_{+}+B\mathbb{P}_{-}\\
=&B(B^{-1}A\mathbb{P}_{+}+\mathbb{P}_{-})\\
=&B(\mathbb{P}_{+}B^{-1}A\mathbb{P}_{+}+\mathbb{P}_{+}B^{-1}A\mathbb{P}_{+}\mathbb{P}_{-}B^{-1}A\mathbb{P}_{+}+\mathbb{P}_{-}B^{-1}A\mathbb{P}_{+}
+\mathbb{P}_{-})\\
=&B(\mathbb{P}_{+}B^{-1}A\mathbb{P}_{+}(I+\mathbb{P}_{-}B^{-1}A\mathbb{P}_{+})+\mathbb{P}_{-}(\mathbb{P}_{-}B^{-1}A\mathbb{P}_{+}
+I))\\
=&B(\mathbb{P}_{+}B^{-1}A\mathbb{P}_{+}+\mathbb{P}_{-})(\mathbb{P}_{-}B^{-1}A\mathbb{P}_{+}+I)
\end{align*}
where $I+\mathbb{P}_{-}B^{-1}A\mathbb{P}_{+}$ is invertible on , the inverse is $I-\mathbb{P}_{-}B^{-1}A\mathbb{P}_{+}.$
This implies
\begin{align}\label{q1}
A\mathbb{P}_{+}+B\mathbb{P}_{-}{\sim} \mathbb{P}_{+}B^{-1}A\mathbb{P}_{+}+\mathbb{P}_{-}.
\end{align}
Moreover, under the decomposition $L^{2}_{2}(\mathbb{T})=H^{2}_{2}(\mathbb{T})\oplus (H^{2}(\mathbb{T}))^{\bot}_{2},$ we have
\begin{align*}
\mathbb{P}_{+}B^{-1}A\mathbb{P}_{+}+\mathbb{P}_{-}
=\begin{pmatrix}T_{B^{-1}A}& 0\\
	0 &   I_{(H^{2}(\mathbb{T}))^{\bot}_{2}}\end{pmatrix},
\end{align*}
where $T_{B^{-1}A}$ is a block Toeplitz operator on $H^{2}_{2}(\mathbb{T})$ and
\begin{align}\label{BA}
B^{-1}A
=\begin{pmatrix}g\psi^{-1}& -\psi^{-1}\\
		g\phi\psi^{-1} -f &   -\phi\psi^{-1} \end{pmatrix},
\end{align}
$\det B^{-1}A=-f\psi^{-1}.$
Hence,
\begin{align}\label{q2}
\mathbb{P}_{+}B^{-1}A\mathbb{P}_{+}+\mathbb{P}_{-}\stackrel{\ast}{\sim} T_{B^{-1}A}.
\end{align}

Similarly,
\begin{align*}
&A\mathbb{P}_{+}+B\mathbb{P}_{-}
=A(\mathbb{P}_{+}+\mathbb{P}_{-}A^{-1}B\mathbb{P}_{-})(\mathbb{P}_{+}A^{-1}B\mathbb{P}_{-}+I)
\end{align*}
This implies
\begin{align*}
A\mathbb{P}_{+}+B\mathbb{P}_{-}{\sim} \mathbb{P}_{+}+\mathbb{P}_{-}A^{-1}B\mathbb{P}_{-}.
\end{align*}
and
\begin{align*}
\mathbb{P}_{+}+\mathbb{P}_{-}A^{-1}B\mathbb{P}_{-}\stackrel{\ast}{\sim} \mathbb{J}T_{(A^{-1}B)^{*}}\mathbb{J}.
\end{align*}
where $\mathbb{J}(f,f)^{T}=(Jf,Jf)^{T}
=(\bar{z}\bar{f},\bar{z}\bar{f})^{T}$ for $f\in L^{2}(\mathbb{T}).$

Recall the invertibility and Fredholm of Toeplitz operators with matrix-symbols via Wiener-Hopf factorization.
\begin{definition}
A representation of the form $F=F_{-}DF_{+}$ is called Winer-Hopf factorization of the invertible matrix function $F\in L^{\infty}_{N\times N}(\mathbb{T})$
if $D=diag(z^{\kappa_{j}})^{N}_{j=1}$ with $\kappa_{j}\in\mathbb{Z},$ and if $F_{-}$ and $F_{+}$ satisfy the following conditions:
\begin{enumerate}
\item $F_{+},F^{-1}_{+} \in H_{N \times N}^{2}(\mathbb{T}),$ $F_{-},F^{-1}_{-} \in \overline{H_{N \times N}^{2}(\mathbb{T})},$
\item The operator $F^{-1}_{+}\mathbb{P}_{+}F_{+}$ is defined on the linear space of all $\mathbb{C}^{N}$-valued trigonometric polynomials, can be extended to a bounded operator on $H^{2}_{N}(\mathbb{T}).$
\end{enumerate}
\end{definition}

\begin{theorem}\cite{Simonenko}\label{infred}
Let $F \in L_{N \times N}^{\infty}(\mathbb{T})$. Then $T_{F}$ is invertible(resp. Fredholm) if and only if $F$ admits a Wiener-Hopf factorization.$F=F_{-}F_{+}$(resp. $F=F_{-}DF_{+}$).

If $T_a$ is Fredholm, then
\begin{align*}
\operatorname{dim} \operatorname{Ker} T_a=-\sum_{\kappa_{j}<0} \kappa_{j}, \quad \operatorname{dim} \operatorname{Coker} T_a=\sum_{\kappa_{j}>0} \kappa_{j} .
\end{align*}
\end{theorem}

\begin{theorem}
If $H={\tiny\begin{pmatrix}
   f& \phi\\
   g & \psi \end{pmatrix}}\in L_{2\times2}^{\infty}(\mathbb{T}),$ then $R_H$
is invertible (resp. Fredholm) if and only if
$f$ and $\psi$ are invertible in $L^{\infty}(\mathbb{T})$ and $\tiny\begin{pmatrix}g\psi^{-1}& -\psi^{-1}\\
		g\phi\psi^{-1} -f &   -\phi\psi^{-1} \end{pmatrix}$
admit a Winer-Hopf factorization $F_{-}F_{+}$(resp.$F_{-}DF_{+}$).

If $R_H$ is Fredholm, then
\begin{align*}
\operatorname{dim} \operatorname{Ker}R_H=-\sum_{k_{j}< 0} k_{j},
\quad  \operatorname{dim} \operatorname{Ker}R^{*}_H=\sum_{k_{j}> 0} k_{j}.
\end{align*}
\end{theorem}
\begin{proof}
If $R_H$ is invertible or Fredholm,
by Corollary \ref{fredholm}, we have $f$ and $\psi$ are invertible in $L^{\infty}(\mathbb{T}).$
Since the relation $\stackrel{*}{\sim}$ is transitive, combining Theorem \ref{q}, \eqref{q1}
and \eqref{q2},
it follows that $R_{H}\stackrel{*}{\sim}T_{B^{-1}A}.$
Using Theorem \ref{iff} and Theorem \ref{infred}, we get the result.
\end{proof}
\section{Applications}

\subsection{The Spectral Inclusion Theorem}

In the theory of Toeplitz operator, the spectrum of $T_{\phi}$ always includes
the essential range of $\phi.$ Corollary \ref{fredholm} shows that
\begin{align*}
\emph{Ran}_{ess}{f} \cup \emph{Ran}_{ess}{\psi}\subset \sigma(R(f,g,\phi,\psi)).
\end{align*}
Hence, for the bounded singular integral operator $R_{\alpha,\beta},$ we have
\begin{align*}
\emph{Ran}_{ess}{\alpha} \cup \emph{Ran}_{ess}{\beta}\subset \sigma(R_{\alpha,\beta}),
\end{align*}
for the bounded dual truncated Toeplitz operator $D_{\phi},$ we have
\begin{align*}
\emph{Ran}_{ess}{\phi} \subset \sigma(D_{\phi});
\end{align*}
for the bounded Foguel-Hankel operator
$
\begin{pmatrix}T_{z}^{*}& X\\
		0 &   T_{z}\end{pmatrix},$ we have
\begin{align*}
\mathbb{T} \subset \sigma \begin{pmatrix}T_{z}^{*}& X\\
		0 &   S\end{pmatrix}.
\end{align*}
Moreover, for every constant $\lambda,$ we have
\begin{align*}
\lambda I-\begin{pmatrix}T_{z}^{*}& X\\
		0 &   T_{z}\end{pmatrix}
=\begin{pmatrix}I& 0\\
		0 &   \lambda I-T_{z}\end{pmatrix}
\begin{pmatrix}I& -X\\
		0 &   I\end{pmatrix}
\begin{pmatrix}\lambda I-T_{z}^{*}& 0\\
		0 &   I\end{pmatrix}.
\end{align*}
Note that
\begin{align*}
\begin{pmatrix}I& -X\\
		0 &   I\end{pmatrix}
\end{align*} is always invertible and
\begin{align*}
\begin{pmatrix}I& -X\\
		0 &   I\end{pmatrix}^{-1}=\begin{pmatrix}I& X\\
		0 &   I\end{pmatrix}.
\end{align*}
If both of $\lambda I-T_{z}$ and $\lambda I-T_{z}^{*}$ are
invertible, then $\lambda I-\begin{pmatrix}T_{z}^{*}& X\\
		0 &   T_{z}\end{pmatrix}$ is invertible.
Therefore,\begin{align*}
\sigma \begin{pmatrix}T_{z}^{*}& X\\
		0 &   T_{z}\end{pmatrix}\subset \sigma(T_{z})=\overline{\mathbb{D} }.
\end{align*}
\subsection{Essential spectrum}
The essential spectrum of Toeplitz operator with continous symbol equals the essential range of the symbol.
Corollary \ref{QC} shows that
If $f,g,\phi,\psi\in C(\mathbb{T}),$ then \[\sigma_{e}(R(f,g,\phi,\psi))=f(\mathbb{T}) \cup \psi(\mathbb{T}).\]
Hence, for bounded singular integral operator, if $\alpha,\beta\in C(\mathbb{T}),$ then
\[\sigma_{e}(R_{\alpha,\beta})=\alpha(\mathbb{T}) \cup \beta(\mathbb{T}).\]
For bounded dual truncated Toeplitz operator, if $\varphi\in C(\mathbb{T}),$ then
\[\sigma_{e}(D_{\varphi})=\varphi(\mathbb{T}).\]
For bounded Foguel-Hankel operator, if $X=\Gamma_{\phi}$ and $\phi\in H^{\infty}+C(\mathbb{T}),$
then
\begin{align*}
\sigma_{e} \begin{pmatrix}T_{z}^{*}& X\\
		0 &   T_{z}\end{pmatrix}=\mathbb{T}.
\end{align*}

\subsection{Special cases}
We consider  one of operators $T_{f}$ and $S_{\psi}$ is invertible. In particular, suppose $S_{\psi}=I,$
Suppose that $\lambda \notin \emph{Ran}_{ess}{f} \cup \{1\}.$
Now
\begin{gather*}
\begin{pmatrix}T_{f-\lambda} & H^{*}_{\bar {\phi}}\\
		H_{g} &   I-\lambda \end{pmatrix}
=\begin{pmatrix}I & H^{*}_{\bar {\phi}}\\
		0 &   \frac{1}{1-\lambda }I \end{pmatrix}
\begin{pmatrix}T_{f-\lambda}-H^{*}_{\bar{\phi}}H_{g} & 0\\
		0 &   I \end{pmatrix}
\begin{pmatrix}I & 0\\
		(1-\lambda )H_{g} &   I \end{pmatrix}.
\end{gather*}
Since $\begin{pmatrix}I & H^{*}_{\bar {\phi}}\\
		0 &   \frac{1}{1-\lambda }I \end{pmatrix}$
and $\begin{pmatrix}I & 0\\
		(1-\lambda )H_{g} &   I \end{pmatrix}$ are always
invertible, or
$\begin{pmatrix}I & H^{*}_{\bar {\phi}}\\
		0 &   \frac{1}{1-\lambda }I \end{pmatrix}^{-1}
=\begin{pmatrix}I & -(1-\lambda)H^{*}_{\bar {\phi}}\\
		0 &   (1-\lambda)I \end{pmatrix}$
and
$\begin{pmatrix}I & 0\\
		(1-\lambda )H_{g} &   I \end{pmatrix}^{-1}
=\begin{pmatrix}I & 0\\
		-(1-\lambda )H_{g} &   I \end{pmatrix},$
it follows that \[\begin{pmatrix}T_{f-\lambda} & H^{*}_{\bar {\phi}}\\
		H_{g} &   I-\lambda \end{pmatrix}\] is invertible
if and only if
\[\begin{pmatrix}T_{f-\lambda}-H^{*}_{\bar{\phi}}H_{g} & 0\\
		0 &   I \end{pmatrix}\]
is invertible.
Therefore, we have
\[\sigma(R(f,g,\phi,1))=\sigma(T_{f}-H^{*}_{\bar{\phi}}H_{g})\cup \emph{Ran}_{ess}{f} \cup \{1\}.\]
Since $T_{f}-H^{*}_{\bar{\phi}}H_{g}=T_{f}-T_{\phi g}+T_{\phi}T_{g},$ we have
\begin{align*}
\lim_{r\rightarrow 1^{-}}\langle T_{f}-H^{*}_{\bar{\phi}}H_{g}k_{r\xi},k_{r\xi}\rangle=f(\xi)
\quad a.e. ~on ~\mathbb{T}.
\end{align*}
By Corollary \ref{fredholm}, we have $ \emph{Ran}_{ess}{f} \subset \sigma(T_{f}-H^{*}_{\bar{\phi}}H_{g}).$
Hence,\[\sigma(R(f,g,\phi,1))=\sigma(T_{f}-H^{*}_{\bar{\phi}}H_{g})\cup \{1\}.\]
Similarly,
\[\sigma(R(1,g,\phi,\psi))=\sigma(S_{\psi}-H_{g}H^{*}_{\bar {\phi}})\cup \{1\}.\]

\end{document}